\documentclass[10pt]{article}
\pdfoutput=1
\usepackage{amsmath,amsthm,amssymb}
\usepackage[T1]{fontenc}
\usepackage[utf8x]{inputenc}
\usepackage[dvipdf]{graphicx}

\usepackage{color}
\usepackage{graphics}
\usepackage{epstopdf}
\usepackage{enumerate}
\usepackage{enumitem}
\usepackage[normalem]{ulem}
\usepackage{amsmath}
\usepackage{cancel}
\write18{RoughHeston}
\definecolor{db}{RGB}{0, 0, 130}
\usepackage[colorlinks=true,citecolor=red,linkcolor=db,urlcolor=blue,pdfstartview=FitH]{hyperref}

\usepackage[numbers]{natbib}

\definecolor{rp}{rgb}{0.25, 0, 0.75}
\definecolor{dg}{rgb}{0, 0.5, 0}

\newcommand{\R}{\mathbb{R}}

\newcommand{\N}{\mathbb{N}}
\newcommand{\EE}{\mathbb{E}}

\newcommand{\ens}{\eta_n(s)}
\newcommand{\ent}{\eta_n(t)}
\newcommand{\enu}{\eta_n(u)}
\newcommand{\enr}{\eta_n(r)}

\makeatletter
\newcommand{\customlabel}[2]{%
   \protected@write \@auxout {}{\string\newlabel {#1}{{#2}{\thepage}{#2}{#1}{}}}%
   \hypertarget{#1}{#2\hspace{-0.13cm}}
}
\makeatother

\newtheorem{theorem}{Theorem}[section]

\newtheorem{definition}{Definition}[section]

\newtheorem{corollary}[definition]{Corollary}

\newtheorem{assumption}[theorem]{Assumption}
\newtheorem{lemma}[definition]{Lemma}

\newtheorem{proposition}[definition]{Proposition}
\newtheorem{remark}[definition]{Remark}

\def\E{\mathbb{E}}

\def\Om{\Omega}
\def\Fc{\mathcal{F}}
\def\F{\mathbb{F}}
\def\P{\mathbb{P}}
\def\Xb{\overline{X}}
\def\eps{\varepsilon}

\def\xbf{\mathbf{x}}
\def\zbf{\mathbf{z}}

\def\Lc{\mathcal{L}}

\author{Alexandre Richard\footnote{Universit\'e Paris-Saclay, CentraleSup\'elec, MICS and CNRS FR-3487, France. \texttt{alexandre.richard@centralesupelec.fr}.}
\and Xiaolu Tan\footnote{Department of Mathematics, The Chinese University of Hong Kong. \texttt{xiaolu.tan@cuhk.edu.hk}, research supported by CUHK startup grant and CUHK Faculty of Science Direct Grant 2020-2021.}
\and Fan Yang\footnote{Department of Mathematics, The Chinese University of Hong Kong. \texttt{fyang@math.cuhk.edu.hk}.}
}

\title{On the discrete-time simulation of the rough Heston model}

\begin{document}

\maketitle

\begin{abstract}
	We study Euler-type discrete-time schemes for the rough Heston model,  
	which can be described by a stochastic Volterra equation (with non-Lipschtiz coefficient functions), or by an equivalent integrated variance formulation.
	Using weak convergence techniques, we prove that the limits of the discrete-time schemes are solution to some modified Volterra equations. 
	Such modified equations are then proved to share the same unique solution as the initial equations, which implies the convergence of the discrete-time schemes.
	Numerical examples are also provided in order to evaluate different options' prices under the rough Heston model.

	\medskip


	\medskip

	\noindent\textbf{MSC2020 subject classification:}  60H20, 45D05, 91G60.

\end{abstract}

\section{Introduction}

	The modelling of rough volatilities is an important subject in mathematical finance,
	especially since the paper \cite{gatheral2018volatility} which brought statistical evidence of such behaviour in the financial markets (see also the pioneering works \cite{comte1998long, alos2007short, fukasawa2011asymptotic, comte2012affine} in this direction).
	Rough volatility models have the advantage to better exhibit the roughness of the volatility time series, to reproduce the shape and the dynamic of the implied volatility surface, etc.
	Many of the rough volatility models consist in replacing the Brownian motions in the classical models by fractional Brownian motions, which leads to some SDEs or more general stochastic system driven by fractional Brownian motions, see e.g. \cite{comte1998long,  comte2012affine,  bayer2016pricing, akahori2017probability, gatheral2018volatility, araneda2020fractional} among many others.
	Another important way to model the rough volatility process is to use a stochastic Volterra equation, 
	such as the  rough Heston model introduced by El Euch and Rosenbaum \cite{el2019characteristic}:
	\begin{equation}\label{eq:RoughHeston_intro}
		S_t = S_0+\int_0^t S_s \sqrt{V_s} ~d W^1_s ,
		~~~~
		V_t=V_0+\int_0^t K(t-s) \Big( (\theta -\lambda V_s)~ds+ \nu\, \sqrt{V_s}~dW^2_s \Big) ,
	\end{equation}
	where $(W^1, W^2)$ are two correlated Brownian motions with some correlation constant $\rho \in (-1, 1)$, and $K(t) := C t^{H-\frac{1}{2}}$ is the kernel function with some Hurst parameter $H\in (0,\frac{1}{2})$.
	Above, $S$ is the risky asset price process under the risk neutral probability, and $V$ represents the volatility process.
	Further, by considering the integrated processes: 
	$$
		X_t := \int_0^t V_s~ds,
		~~~
		M^1_t := \int_0^t \sqrt{V_s} dW^1_s,
		~~~
		M^2_t := \int_0^t \sqrt{V_s} dW^2_s,
	$$
	the rough Heston model \eqref{eq:RoughHeston_intro} is shown to be equivalent (see e.g. \citet{abi2021weak}) to the following system
	\begin{equation}\label{eq:HyperRough_intro}
			S_t = S_0+\int_0^t  S_s~d M^1_{s},
			~~~~
			X_t = V_0 t+\int_0^t K(t-s) \big(\theta s -\lambda X_s+\nu M^2_s \big)~ds,
	\end{equation}
	where $(M^1,M^2)$ are two continuous martingales with quadratic variation $\langle M^1 \rangle =\langle M^2 \rangle= X$ and quadratic covariation $\langle M^1, M^2 \rangle = \rho X$.
	Based on the above formulations in \eqref{eq:RoughHeston_intro} and \eqref{eq:HyperRough_intro},
	 the super-rough Heston model \cite{dandapani2021quadratic} and the hyper-rough Heston model \cite{abi2021weak,jusselin2020no} have also been developed recently.

	\vspace{0.5em}

	In the rough volatility literature, an important topic is to find a good approximation method (e.g. in order to evaluate option prices), whenever a closed formula is not available.
	For different models, some approximation and asymptotic methods have been introduced and studied, see e.g. \cite{alos2007short, fukasawa2011asymptotic, forde2017asymptotics, horvath2017functional, guennoun2018asymptotic,  abi2019multifactor, friz2020short, horvath2020volatility,  forde2021small}, etc.
	In the meantime, for affine models such as \eqref{eq:RoughHeston_intro} and \eqref{eq:HyperRough_intro}, one can in fact 
	obtain the marginal distributions of $(S_t, V_t)$ or $(S_t, X_t)$ at any time $t \ge 0$,
	by computing their characteristic function via Riccati-type systems, see e.g. \cite{alos2014closed, abi2021weak, abi2019affine, el2019characteristic}.
	This allows in particular to compute efficiently the European call/put options prices under the affine rough volatility models.
	Nevertheless, as the path distribution of the process $(S,V)$ or $(S,X)$ is still unknown, one cannot compute prices of path-dependent options this way.
	In this case, a natural and simple solution would be the Monte-Carlo method based on a discrete-time scheme.

	\vspace{0.5em}
	
	The main objective of the paper is to study the Euler-type discrete-time scheme for the rough Heston model in both formulations \eqref{eq:RoughHeston_intro} and \eqref{eq:HyperRough_intro}, and to provide a convergence result.
	We will stay in a more general rough Heston setting, i.e. the kernel function $K(t)$ is not necessarily of the form $C t^{H-\frac{1}{2}}$.
	Throughout the paper, we would like to call \eqref{eq:RoughHeston_intro} the rough Heston model in the stochastic Volterra equation formulation, and \eqref{eq:HyperRough_intro} the rough Heston model in the integrated variance formulation (or simply integrated-rough Heston model).

	\vspace{0.5em}

	Notice that Equation \eqref{eq:RoughHeston_intro} satisfied by $V$ is a standard stochastic Volterra equation (but with non-Lipschitz coefficient).
	For stochastic Volterra equations with Lipschitz coefficient equations, the discrete-time schemes such as Euler scheme and/or Milstein scheme have been studied in \cite{zhang2008euler, richard2020discrete,  AlfonsiKebaier,10.1093/imanum/drab047}, where (sharp) strong convergence rates have been obtained.
	Nevertheless, because of the square root term $\sqrt{V_s}$, the coefficient function in  \eqref{eq:RoughHeston_intro} is non-Lipschitz. Hence the techniques and results in the aforementioned papers cannot be applied to obtain a convergence result.

	\vspace{0.5em}

	We will apply weak convergence techniques to provide a convergence proof of the discrete-time scheme for both Equations \eqref{eq:RoughHeston_intro} and \eqref{eq:HyperRough_intro}.
	The idea is very classical in the literature on SDEs, see e.g. \cite{karatzas1998brownian,  jacod2013limit}.
	First, one shows that the sequence of discrete time numerical solutions is tight, then that any limit of the sequence is solution to the continuous-time equation.
	Next, it is enough to show that the limit continuous-time equation has a unique weak solution, 
	so that the numerical solution converges weakly to the unique solution of the limit equation.
	In the context of the rough Heston model, such weak convergence techniques have already been used in \cite{abi2021weak, abi2019weak, abi2019affine}, in particular to show the existence of weak solutions of the related equations.
	
	\vspace{0.5em}
	
	However, for the analysis of the discrete-time numerical solution, it is not straightforward to apply their techniques and results.
	First, they usually consider sequences of continuous-time processes $(V^n)_{n \ge 1}$ and $(X^n)_{n \ge 1}$ which are solutions of equations with smoother coefficients, where generalized Gr\"onwall lemma applies under conditions on the kernel $K$.
	For the discrete-time numerical solution, because of discretization of the kernel function $K$, 
	it is not trivial to formulate explicit conditions on $K$ to have the discrete-time generalized Gr\"onwall lemma.
	We therefore need to develop different techniques to estimate the (uniform) moment estimates to obtain the tightness of discrete-time solutions.
	Next, their approximating processes are already positive (resp. non-decreasing), so that the limit process $V$ (resp. $X$) is automatically positive (resp. non-decreasing).
	In the discrete-time setting, the numerical solution $V^n$ may not always be positive, we hence need  to take its positive part $(V^n)_+$ before taking the square root, i.e. $\sqrt{(V^n)_+}$.
	As for the numerical solution $X^n$, we need to replace $X^n_t$ by $\Xb^n_t := \max_{0 \le s \le t} X^n_s$ to make it non-decreasing, so that it can be the quadratic variation of some martingales.
	Consequently, it turns out that the limit of the numerical solutions is solution to some modified equations.
	We will then need to show that the limit process $V$ is positive and $X$ is non-decreasing, 
	and the modified equation shares the same unique weak solution as the initial equations.
	For this, we will adapt the ideas from \cite{abi2019markovian} to our context.
	This allows us to obtain weak convergence results of the discrete-time numerical solutions.
	Finally, we also provide some numerical simulation examples to evaluate (path-dependent) options' prices in the rough Heston model.

	\vspace{0.5em}

	The rest of the paper is organised as follows. 
	In Section \ref{sec:results}, we state the two equivalent formulations of the rough Heston model with more details, 
	and present the corresponding discrete-time schemes as well as the convergence results. 
	In Section \ref{sec:NumericalExamples}, we provide some numerical examples in order to evaluate the option prices in the rough Heston model. 
	Proofs of the main (weak) convergence results in Theorems \ref{thm:CvgRoughHeston} and \ref{thm:CvgIntegVariance} are provided in Section \ref{sec:proofs}.

\section{Discrete-time simulation of the rough Heston model}
\label{sec:results}

	We will first restate the two equivalent formulations \eqref{eq:RoughHeston_intro} and \eqref{eq:HyperRough_intro} of the rough Heston model with more precise definitions.
	Based on the two formulations, we introduce the corresponding Euler-type schemes, which are defined on a discrete grid.
	Let us consider a sequence $(\pi_n)_{n \ge 1}$ of discrete-time grid on $[0,T]$, 
	with $\pi_n=\{0=t_0^n <t_1^n <t_2^n <...<t_n^n=T\}$ for each $n \ge 1$.
	For each $n \ge 1$, we define $\ens := t^n_k$ for $s \in [t^n_k, t^n_{k+1})$, $k=0, \cdots, n-1$, and $\eta_n(T) := T$.
	
	\vspace{0.5em}

	To provide the convergence results, we will make some assumptions on the kernel function $K: [0,T]  \longrightarrow \R$ used in the model.
	For a (measurable) kernel function $K: [0,T] \longrightarrow \R$, let us recall that the resolvent of the first kind of $K$ is a finite measure $L$ on $[0,T]$ such that 
	\begin{equation} \label{eq:def_resolvent}
		(K*L) (t) 
		:= \int_{[0,t]} K(t-s) L(ds) = 1,
		~~\mbox{for all}~t \in (0,T].
	\end{equation}
	Notice that we are in a one-dimensional context, so that the above definition is much simpler than the general one 
	(see  \cite{gripenberg1990volterra}).

	\begin{assumption} \label{assum:main}
		The function $K \in L^2([0,T])$ is nonnegative, not identically $0$, non-increasing and continuous. Its resolvent of the first kind $L$ is nonnegative and such that $s \longmapsto L([s, s+t])$ is non-increasing for all $t \ge 0$.
		Moreover, there exist constants $C>0$ and $H >0$ such that, for all $0 \le t \le T$, and $n \ge 1$, $\delta\in(0, T-t]$, 
		one has	
		\begin{equation} \label{eq:A2}
			\int_t^{t + \delta}  \big| K(t + \delta - \ens) \big|^2~ds
			~\leq~
			C \delta^{2 H },
		\end{equation}
		and
		\begin{equation} \label{eq:A3}
			\int_0^{t} \big| K(t+\delta-\ens)-K(t-\ens) \big|^2~ds
			~\leq~
			C \delta^{2 H}.
		\end{equation}
	\end{assumption}
	
	\begin{remark}
		$\mathrm{(i)}$ Let $K(t) := C t^{H-\frac{1}{2}}$ with Hurst constant  $H \in (0, \frac{1}{2})$ and some constant $C > 0$.
		Then  \eqref{eq:A2} and \eqref{eq:A3} can be checked by direct computation, while the resolvent of $K$ is $L(dt) = C_{H}\, t^{-(H+\frac{1}{2})}\, dt$ for some $C_{H}>0$, and thus satisfies Assumption \ref{assum:main}.
		
%
%
%

		\vspace{0.5em}
		
		\noindent $\mathrm{(ii)}$ Let $K(t) := K_1(t) K_2(t)$, where $K_1(t) = C \exp (-\beta t)$ and $K_2(t) = t^{H-\frac12}$ for some constants $C > 0$, $\beta > 0$ and $H \in (0, \frac12)$.
		One can check by direct computation that $K$ satisfies \eqref{eq:A2} and \eqref{eq:A3}.
		Moreover,  the resolvent of $K$ is explicitely given in \cite[Table 1]{abi2019affine} and satisfies Assumption \ref{assum:main}.
		
		\vspace{0.5em}
	
		\noindent $\mathrm{(iii)}$ More generally, if $K$ is a  completely monotone function (as defined in \cite[Section 5.2]{gripenberg1990volterra}) and is not identically $0$, then by \cite[Theorem 5.5.4]{gripenberg1990volterra} it admits a resolvent of the first kind $L$ which is nonnegative and is such that $s \longmapsto L([s, s+t])$ is non-increasing for all $t \ge 0$. So if $K$ also verifies \eqref{eq:A2} and \eqref{eq:A3}, it will satisfy Assumption \ref{assum:main}. This is for instance the case of $K_{1}$ and $K_{2}$ above, as well as for instance $K_{3}(t) = \log(1+\frac{1}{t+1})$. In addition, any linear combination and multiplication of completely monotone function is still a completely monotone function. Hence Assumption \ref{assum:main} covers a wide range of kernels. 

	\end{remark}

\subsection{The rough Heston model in two equivalent formulations}

	Let $\rho \in [-1, 1]$ be a constant, $S_0$, $V_0$, $\theta$, $\lambda$ and $\nu$ be all strictly positive constants.
	The first formulation of the rough Heston model is given by
	\begin{equation} \label{eq:RoughHeston}
	\begin{split}
		S_t&=S_0+\int_0^t S_s\, \sqrt{V_s}~d \big( \rho\, W_s+\sqrt{1-\rho^2}\, W_s^{\bot} \big),\\
		V_t&=V_0+\int_0^t K(t-s) \Big( \big( \theta-\lambda V_s \big)~ds + \nu \sqrt{V_s}~dW_s \Big),
	\end{split}
	\end{equation}
	where $(W, W^{\bot})$ are two independent Brownian motions.
	Namely, $S$ represents the risky asset price under the risk-neutral probability, $\sqrt{V_t}$ is the volatility at time $t \ge 0$.
	We give immediately a precise definition of weak solution to \eqref{eq:RoughHeston}.

	\begin{definition}\label{def:WeakLpSolutionV}
		We say that Equation \eqref{eq:RoughHeston} has a weak solution if there exists a complete filtered probability space $(\Omega,\mathcal{F},(\mathcal{F}_t)_{t\in [0,T]},\mathbb{P})$ equipped with two independent Brownian motion $W,~W^{\bot}$, and a pair of $\mathbb{R}_+$-valued adapted continuous processes $(V,S)$ such that \eqref{eq:RoughHeston} is satisfied a.s. for all $t \ge 0$.
	\end{definition}

	\begin{remark} \label{rem:uniqueness_RoughHeston}
		Under Assumption \ref{assum:main}, the existence and uniqueness of the weak solution to \eqref{eq:RoughHeston} can be found in e.g. \cite{abi2019affine}.
		Indeed, let $n\longrightarrow \infty$ in \eqref{eq:A2}-\eqref{eq:A3}, it follows by Fatou's lemma that
		$$
			\int_t^{t+\delta} \big| K(t+\delta -s) \big|^2 ds + \int_0^t \big| K(t+\delta -s) - K(t-s) \big|^2 ds \le 2 C \delta^{2H},
		$$
		which corresponds to the conditions required in \cite{abi2019affine}.
		
	\end{remark}

	The process $(\sqrt{V_t})_{t \ge 0}$ is called the volatility process. 
	Let us consider the integrated variance process $X$ given by 
	$$
		X_t=\int_0^t V_s ~ds,~~t \ge 0.
	$$
	As observed in \cite{abi2021weak}, one can reformulate \eqref{eq:RoughHeston} into an equivalent system on $(X, S)$.
	Namely, by applying the stochastic Fubini theorem (see e.g. \cite[p.175]{RevuzYor}),
	the processes $(S, X)$ satisfy the following stochastic Volterra equation
	\begin{align}\label{eq:IntegVarHeston}
	\begin{split}
		S_t&=S_0+\int_0^t S_s~d \big(\rho M_s+\sqrt{1-\rho^2}M_s^{\bot} \big),\\
		X_t&=V_0 t +\int_0^t K(t-s)\big(\theta s -\lambda X_s+\nu M_s\big)~ds,
	\end{split}
	\end{align}
	where $M_s,~M^{\bot}_s$ are two orthogonal continuous martingales with quadratic variation $\langle M \rangle  = \langle M^{\bot}\rangle =X$, and initial condition $M_0=M^{\bot}_0 = 0$.
	We will call it the integrated variance formulation (or simply integrated-rough Heston model).
	Following \cite{abi2021weak}, let us introduce the definition of weak solution of Equation \eqref{eq:IntegVarHeston}.

	\begin{definition}\label{def:WeakSolX}
		We say that Equation \eqref{eq:IntegVarHeston} has a weak solution if there exists a filtered probability space $(\Omega,\mathcal{F},(\mathcal{F}_t)_{t\geq 0},\mathbb{P})$ supporting a pair of orthogonal continuous martingales $(M,M^{\bot})$,
		a non-decreasing, non-negative, continuous and adapted process $X$ and a non-negative continuous and adapted process $S$, 
		such that \eqref{eq:IntegVarHeston} holds a.s.
	\end{definition}

	\begin{remark} \label{eq:uniqueness_VarHeston}
		Under Assumption \ref{assum:main}, the existence and uniqueness of the weak solution to Equation \eqref{eq:IntegVarHeston} is proved in \cite{abi2021weak}.
		In fact, the context in \cite{abi2021weak} covers the case with $L^1$ kernel functions.
		For technical reasons, in particular the equivalence results in Section \ref{subsec:equivalence}, we will stay in the context of $L^2$ kernel functions $K$.
	\end{remark}

	\begin{remark} \label{rem:log_transform}
		Let us define $Y_t := \log(S_t)$, then it is clear that one has
		$$
			Y_t = Y_0 - \int_0^t \frac12 V_s ds + \int_0^t \sqrt{V_s}~d \big( \rho\, W_s+\sqrt{1-\rho^2}\, W_s^{\bot} \big)
			~\mbox{in the formulation \eqref{eq:RoughHeston}},
		$$
		and
		$$
			Y_t = Y_0 - \frac12 X_t +  \rho M_t + \sqrt{1-\rho^2}M^{\bot}_t
			~\mbox{in the formulation \eqref{eq:IntegVarHeston}}.
		$$
	\end{remark}

\subsection{The discrete-time schemes and their convergence}
\label{subsec:Schemes}

	Recall that $(\pi_n)_{n \ge 1}$ is a sequence of discrete-time grids on $[0,T]$, 
	with $\pi_n=\{0=t_0^n <t_1^n <t_2^n <...<t_n^n=T\}$ for each $n \ge 1$.
	Let us denote $\delta_n := |\pi_n| := \max\limits_{0\leq k\leq n-1}\Delta t^n_{k+1}$, 
	with $\Delta t^n_{k+1} := t^n_{k+1} - t^n_k$ and $\ens:=t_k^n$, 
	for $s\in [t_k^n,t_{k+1}^n)$, $k = 0, \cdots, n-1$.
	We will simulate $(S, V)$ of \eqref{eq:RoughHeston} and $(S, X)$ of \eqref{eq:IntegVarHeston} on the discrete-time grid $\pi_n$.
	More precisely, in view of Remark \ref{rem:log_transform}, we would like to simulate the process $(Y, V)$ in place of $(S, V)$ in \eqref{eq:RoughHeston}, 
	and to simulate $(Y, X)$ in place of $(S, X)$ in \eqref{eq:IntegVarHeston}.
	As observed in the Black-Scholes model, the simulation of $Y$ permits to avoid the time discretization of the process $S$ in the dynamics of $S$, and one can expect a better performance for its simulation.

	\vspace{0.5em}

	Let us first give the Euler-type scheme for \eqref{eq:RoughHeston}.
	Notice that the process $V$ is $\R_+$-valued in the continuous-time setting, but it could become negative in a discrete-time simulation.
	For this reason, we use $(V_t)_+ := \max(V_t, 0)$ in the square root term $\sqrt{(V_t)_+}$ to define the discrete-time scheme.
	For the discrete grid $\pi_n$, let us write $t^n_k$ as $t_k$ for simplicity,
	and denote by $(S^n, V^n) = (S^n_{t_k}, V^n_{t_k})_{k=0,1, \cdots, n}$ the corresponding numerical solution, which is given as follows: 
	$S^n_{t_k} := \exp(Y^n_{t_k})$, $k=0, 1, \cdots, n$, and
	\begin{equation}\label{eq:SchemeRoughHeston}
	\begin{split}
		Y_{t_k}^n 
		&= 
		Y_{0} 
		+
		\sum\limits_{i=0}^{k-1} \Big( 
			- \frac{1}{2}(V^n_{t_i})_+\, \Delta t^n_{i+1}
			+ \rho\sqrt{(V^n_{t_i})_+} \big(W_{t_{i+1}}-W_{t_i} \big)
			+ \sqrt{1-\rho^2}\sqrt{(V^n_{t_i})_+} \big( W^{\bot}_{t_{i+1}}-W^{\bot}_{t_i} \big) \Big),\\
		V^n_{t_k}
		&=
		V_0
		+
		\sum_{i=0}^{k-1}\Big( 
			K(t_k-t_i)\big(\theta -\lambda (V^n_{t_i})_+\big) \Delta t^n_{i+1} 
			+
			K(t_k-t_i)\nu \sqrt{(V^n_{t_i})_+} \big( W_{t_{i+1}}-W_{t_i} \big)\Big).\\
	\end{split}
	\end{equation}

	For the process $(S^n, V^n) = (S^n_{t_k}, V^n_{t_k})_{k=0,1, \cdots, n}$ defined on the discrete-time grid $\pi_n$,
	one can use linear interpolation to obtain a continuous time process (with continuous paths), which is still denoted by $(S^n, V^n) = (S^n_t, V^n_t)_{t \in [0,T]}$.
	We can provide a weak convergence result of the numerical solution.
	
	\begin{theorem}\label{thm:CvgRoughHeston}
		\noindent $(i)$ For all $p \ge 2$, there exists a constant $C_p  > 0 $ such that
		$$
			\EE \big[ \big|V^n_t \big|^p \big] \le C_p,
			~~~
			\EE \big[ \big|V_t^n-V_s^n \big|^p \big]
			~\leq~
			C_p (t-s)^{p (H \wedge 1)},
			~~\mbox{for all}~ 0\leq s\leq t\leq T ~\mbox{and}~ n\geq 1.
		$$

		\noindent $(ii)$ The processes $(S^n,V^n)$ defined by \eqref{eq:SchemeRoughHeston} converge weakly to $(S,V)$ in $C([0,T],\R)\times C([0,T],\R)$ as $n \longrightarrow \infty$, where $(S,V)$ is the  unique weak solution to the first formulation \eqref{eq:RoughHeston} of the rough Heston model.
	\end{theorem}

	We now consider the discrete-time simulation problem of Equation \eqref{eq:IntegVarHeston}.
	Notice that the process $X$ in the continuous-time setting is a non-decreasing process, which is the quadratic variation process of the martingales $M$ and $M^{\bot}$.
	In discrete-time simulation, $X$ would not be non-decreasing, and for this reason, we will consider its running maximum $\Xb_t := \max_{0 \le s \le t} X_s$ to define the quadratic variation process.
	For each discrete time grid $\pi_n$, let us define $(S^n, X^n) = (S^n_{t_k}, X^n_{t_k})_{k =0,1, \cdots, n}$ as follows:
	$S^n_{t_k} := \exp(Y^n_{t_k})$, $k=0, 1, \cdots, n$, and
	\begin{equation}\label{eq:schemeIntegRough}          
	\begin{split}
		Y^n_{t_k} & =  Y_0 -\frac{1}{2}\overline{X}^n_{t_k} + \rho M^n_{t_k}+\sqrt{1-\rho^2}M^{n\bot}_{t_k}, \\
		X_{t_k}^{n}&=V_0t_k+\sum_{i=0}^{k-1}K(t_k-t_i)\Big(\theta t_i-\lambda \overline{X}^n_{t_i}+\nu M^{n}_{t_i}\Big) \Delta t^n_i,\\
		M^n_{t_k}& = \sum_{i=1}^k \sqrt{\overline{X}^n_{t_i}-\overline{X}^n_{t_{i-1}}}Z_i,
		~~~~
		M^{n\bot}_{t_k} = \sum_{i=1}^k \sqrt{\overline{X}^n_{t_i}-\overline{X}^n_{t_{i-1}}}Z_i^{\bot},
	\end{split}
	\end{equation}
	where $\overline{X}^n_{t_i} := \underset{ 0\leq j\leq i }{\max} X^n_{t_j}$, and $(Z_i,~Z_i^{\bot})_{i \ge 1}$ is a sequence of  i.i.d. random variables with standard Gaussian distribution $\mathcal{N}(0,1)$.

	\vspace{0,5em}
	
	Similarly, one can interpolate the process $(S^n, X^n) = (S^n_{t_k}, X^n_{t_k})_{k =0,1, \cdots, n}$ from the discrete-time grid $\pi_n$
	to obtain a continuous time process (with continuous paths) $(S^n, X^n) = (S^n_t, X^n_t)_{t \in [0,T]}$.

	\begin{theorem}\label{thm:CvgIntegVariance}
		\noindent $(i)$ For each $p \ge 2$, there exists a constant $C_p > 0$  such that
		$$
			\EE \big[ \big| X^n_t \big|^p \big] + \EE \big[ \big| \Xb^n_t \big|^p \big] \le C_p,
			~~~
			\EE \big[ \big|X^n_t-X^n_s \big|^p \big] 
			~\leq~
			C_p  (t-s)^{p(H \wedge 1)},
			~~\mbox{for all}~ 0\leq s\leq t\leq T ~\mbox{and}~ n\geq 1.
		$$

		\noindent $(ii)$ The processes $(S^n,X^n)$ defined by \eqref{eq:schemeIntegRough} converge weakly to $(S,X)$ in $C([0,T],\R)\times C([0,T],\R)$ as $n \longrightarrow \infty$, 
		where $(S,X)$ is the unique weak solution to the integrated variance formulation \eqref{eq:IntegVarHeston} of the rough Heston model.
	\end{theorem}

	\begin{remark}
	$\mathrm{(i)}$ By the weak convergence of $(S^n, X^n)$ to $(S, X)$, one has the convergence $\E [ f(S^n, X^n)] \longrightarrow \E[ f(S, X)]$ for any bounded continuous payoff function $f$ on $(S,X)$.
	Moreover, since $\sup_{n \ge 1} \E[|\Xb^n_T|^p] < \infty$ for any $p \ge 1$, 
	it follows that $\Lc(X^n)$ converges to $\Lc(X)$ under the $p$-Wasserstein distance, for any $p \ge 1$.
	Consequently, one has the convergence $\E [ f(X^n)] \longrightarrow \E[f(X)]$ for any continuous function $f: C([0,T], \R) \longrightarrow \R$ with polynomial growth.
	
	Unfortunately, we do not have a uniform moment estimation on $S^n$ for all $n \ge 1$.
	For the process $S$ in \eqref{eq:RoughHeston}, the moment explosion of $S_t$ has been studied in \cite{gerhold2019moment}.
	This is possible since one can find a Volterra type equation to compute the characteristic function of the marginal distribution of $\log(S)$. 
	It is nevertheless not clear how to elaborate the same technique on the discrete time Euler scheme solution $S^n$.
	We would like to leave this for future research.
	
	\vspace{0.5em}
	
	\noindent $\mathrm{(ii)}$ Our technique does not allow us to obtain a strong convergence rate.
	At the same time, as the processes $V$ and $X$ are not semimartingales, 
	it is not clear how to adapt the error analysis techniques for classical Heston models, such as in \cite{berkaoui2008euler, alfonsi2015affine}, to this rough Heston model context. 
	Moreover, for Equations \eqref{eq:RoughHeston_intro} and \eqref{eq:HyperRough_intro}, because of the singular kernels and square-root-type coefficients, the strong existence and uniqueness of the solution is still an open question.
	It is therefore not surprising that a strong convergence rate is left as an open question.
	\end{remark}

\section{Numerical examples}
\label{sec:NumericalExamples}

	In this section, we provide some numerical examples to evaluate option prices in the rough Heston model with interest rate $r=0$,
	by using the Monte-Carlo method based on the schemes \eqref{eq:SchemeRoughHeston} and \eqref{eq:schemeIntegRough}. 
	Namely, for an option with payoff function $f(S, X)$, one aims to estimate its price given by 
	$$
		\E[f(S, X)].
	$$
	We use the uniform discrete-time grid $\pi_n = \{ 0 = t_0 < \cdots < t_n = T\}$ with $t_k := k \Delta t$, $\Delta t=\frac{T}{n}$ for Schemes \eqref{eq:SchemeRoughHeston} and \eqref{eq:schemeIntegRough}. 
	We use $M$ i.i.d. copies of simulations $(S^{n,m}, V^{n,m})_{m=1, \cdots, M}$ or $(S^{n,m}, X^{n,m})_{m=1, \cdots, M}$ to estimate the option price $\E[ f(S,X) ]$ by the mean value :
	$$
		\overline U_M ~:=~ \frac{1}{M} \sum \limits_{m=1}^M f \big(S^{n,m}_., X^{n,m}_{\cdot} \big).
	$$
	Notice that from the simulations $V^{n,m}$, one can use $X^{n,m}_{t_k} := \sum_{i=0}^k V^{n,m}_{t_i}$ to compute $X^{n,m}$.
	We also compute the empirical standard deviation of the simulations, which (divided by $\sqrt{M}$) can serve as the statistical error, i.e.
	$$
		\Sigma_M ~:=~ \frac{1}{\sqrt{M}} \bigg( \frac1M \sum \limits_{m=1}^M f \big(S^{n,m}_., X^{n,m}_{\cdot} \big)^2 - \Big( \frac1M \sum \limits_{m=1}^M f \big(S^{n,m}_., X^{n,m}_{\cdot} \big) \Big)^2 \bigg).
	$$
	We then use the following interval as confidence interval of the estimation:
	$$
		\big[ \overline U_M - 2 \Sigma_M,  \overline U_M - 2 \Sigma_M \big].
	$$
	
	For the rough Heston model, we choose  the following parameters: $\lambda=0.3$, $\nu=0.3$, $V_0=0.02$, $\theta=0.02$, 
	$\rho=-0.7$, $S_0=1$,
	and the kernel function $K(t) := \Gamma(H+\frac12)^{-1}  t^{H-\frac12}$ with $H=0.1$.
	We will use different discretization parameters $n$ in our simulations. 
	For each example, we set the number of i.i.d. copies $M = 10^5$, and display the mean value, the statistical error as well as the computational time in the tables.

	\vspace{0.5em}
	
	Notice also that in \eqref{eq:schemeIntegRough}, one can compute 
	$\int_0^{t_k} K(t_k - s) \theta s ds$ explicitly instead of approximating it by $\sum_{i=0}^{k-1} K(t_k - t_i) \theta t_i \Delta t^n_i$.
	This will be taken into account in our simulation.

\subsection{Pricing options on the risky asset}

	Let us first consider the European Call options,  Asian options and Lookback options with the following payoff:
	$$
		(S_T - K)_+,
		~~~
		(A_T - K)_+,
		~~~
		(M_T - K)_+,
	$$
	where $K =1$, $A_T := \int_0^T S_t dt$ and $M_T := \max_{0 \le t \le T} S_t$.
	For the Monte-Carlo method, we simply replace $(S_T, A_T, M_T)$ by $(S^{n,m}_T, A^{n,m}_T, M^{n,m}_T)$ in the payoff function to compute the estimations,
	where $A^{n,m}_T := \frac{T}{n}\sum^n_{k=1}S^{n,m}_{t_k}$, $M^{n,m}_T := \max\limits_{0\leq k\leq n} S^{n,m}_{t_k}$.

	\vspace{0.5em}

	For the European call option, one can compute a reference value for $\E[ (S_T- K)_+]$.
	Indeed, as described in \cite{el2019characteristic}, one can compute the characteristic function of $\log(S_T)$ by solving a fractional Riccati equation, and then use the inverse Fourier transform method to compute $\E[ (S_T- K)_+]$.
	With the above parameters of the rough Heston model, we obtain $0.056832$ as reference value for the European call option.
	
	\vspace{0.5em}

	The numerical results are reported in Figures \ref{fig:EuropCall}, \ref{fig:Asian} and \ref{fig:Lookback} (see also respectively Tables \ref{tab:EuropCall}, \ref{tab:Asian} and \ref{tab:Lookback} for the data).
	We can observe the convergence of both schemes as $n$ increases.
	For European call option and Asian option, it seems that the scheme \eqref{eq:schemeIntegRough} based on integrated variance formulation \eqref{eq:IntegVarHeston}  has a better performance.
	But for the lookback option, the scheme \eqref{eq:SchemeRoughHeston} based on the stochastic Volterra equation \eqref{eq:RoughHeston} seems to perform better.

\begin{figure}[htb]        
 \center{\includegraphics[width=10cm]  {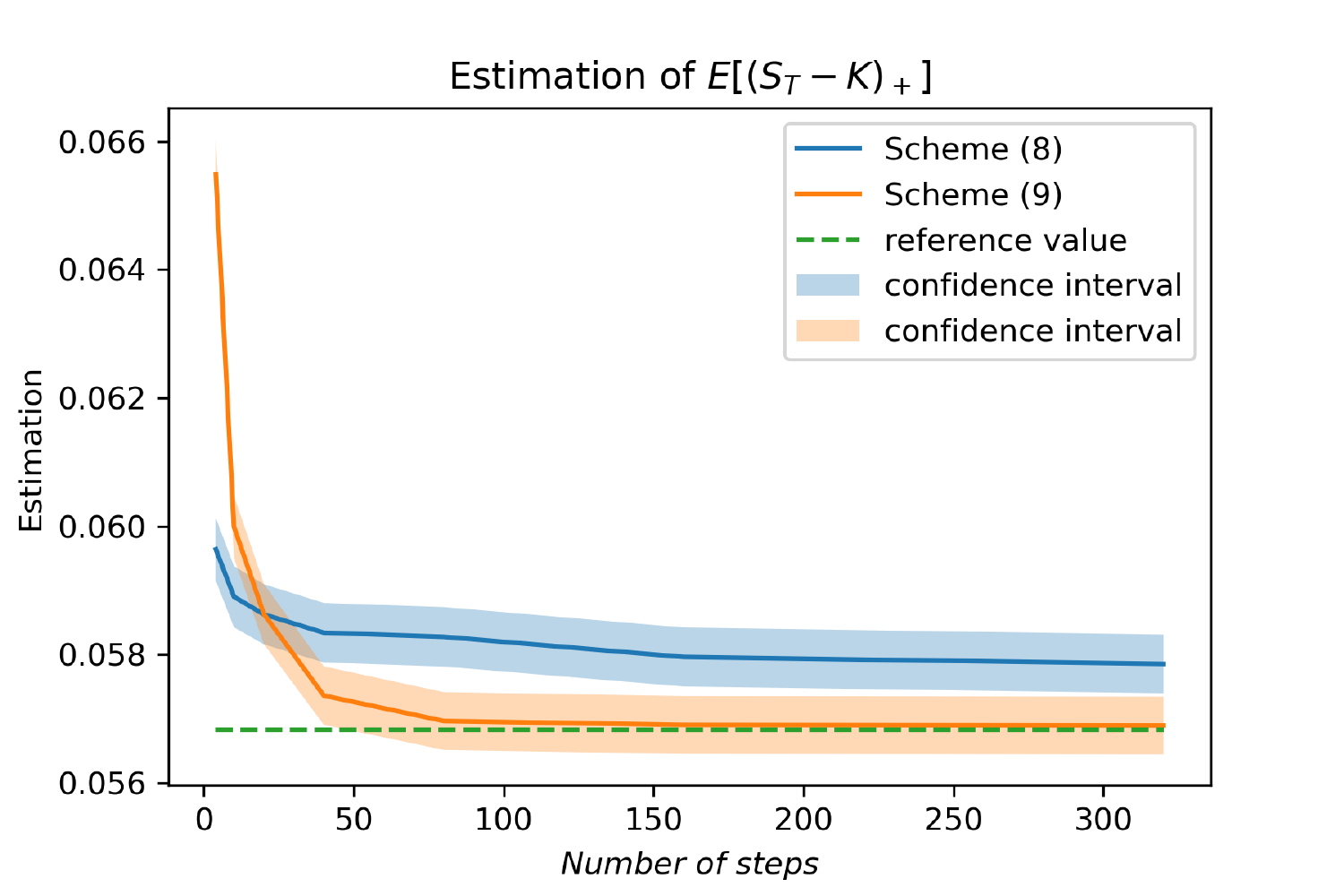}}        
 \caption{\label{fig:EuropCall} Estimation of European call option price}      
 \end{figure}

\begin{figure}[htb]        
 \center{\includegraphics[width=10cm]  {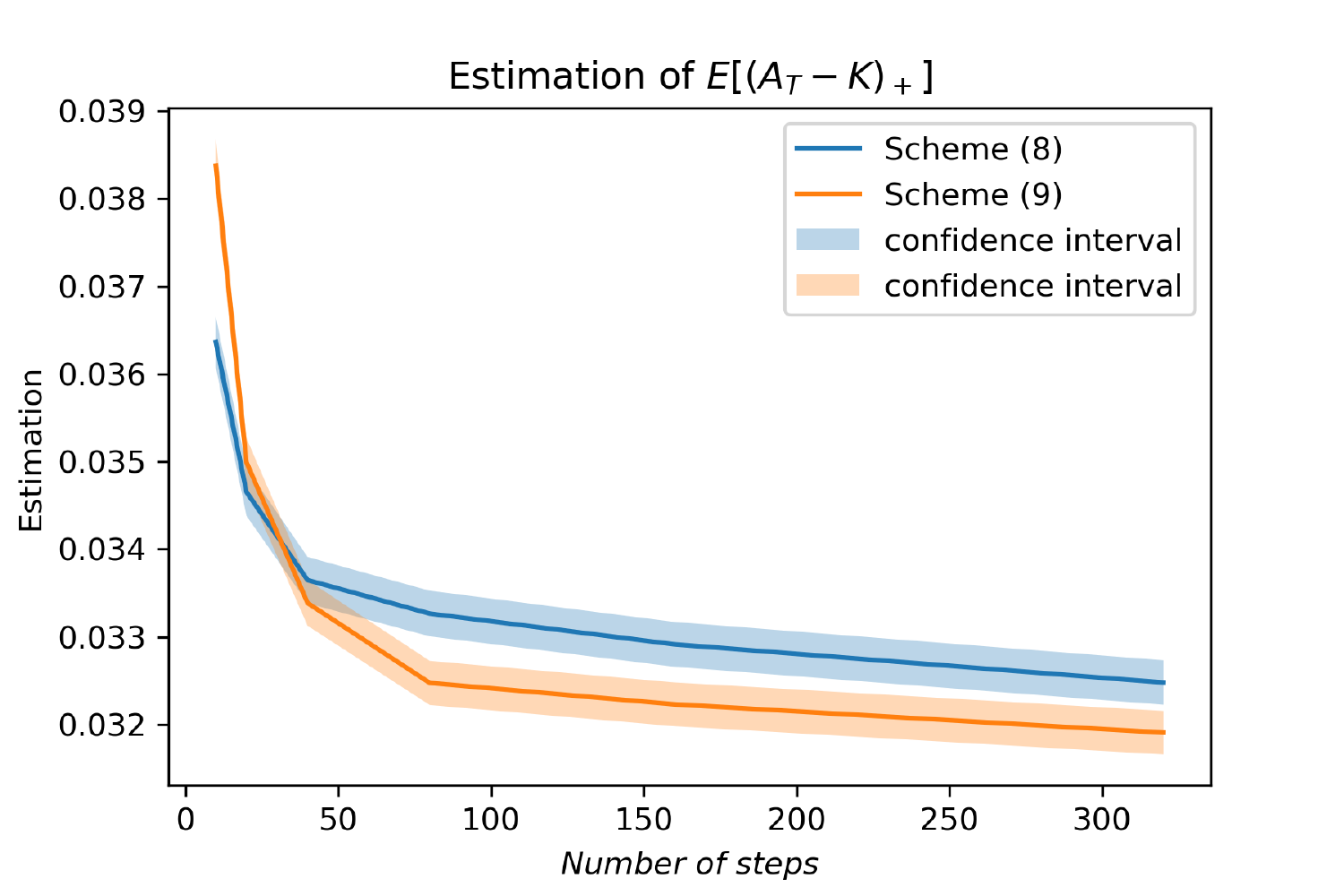}}        
 \caption{\label{fig:Asian} Estimation of Asian option price}      
 \end{figure}

\begin{figure}[htb]        
 \center{\includegraphics[width=10cm]  {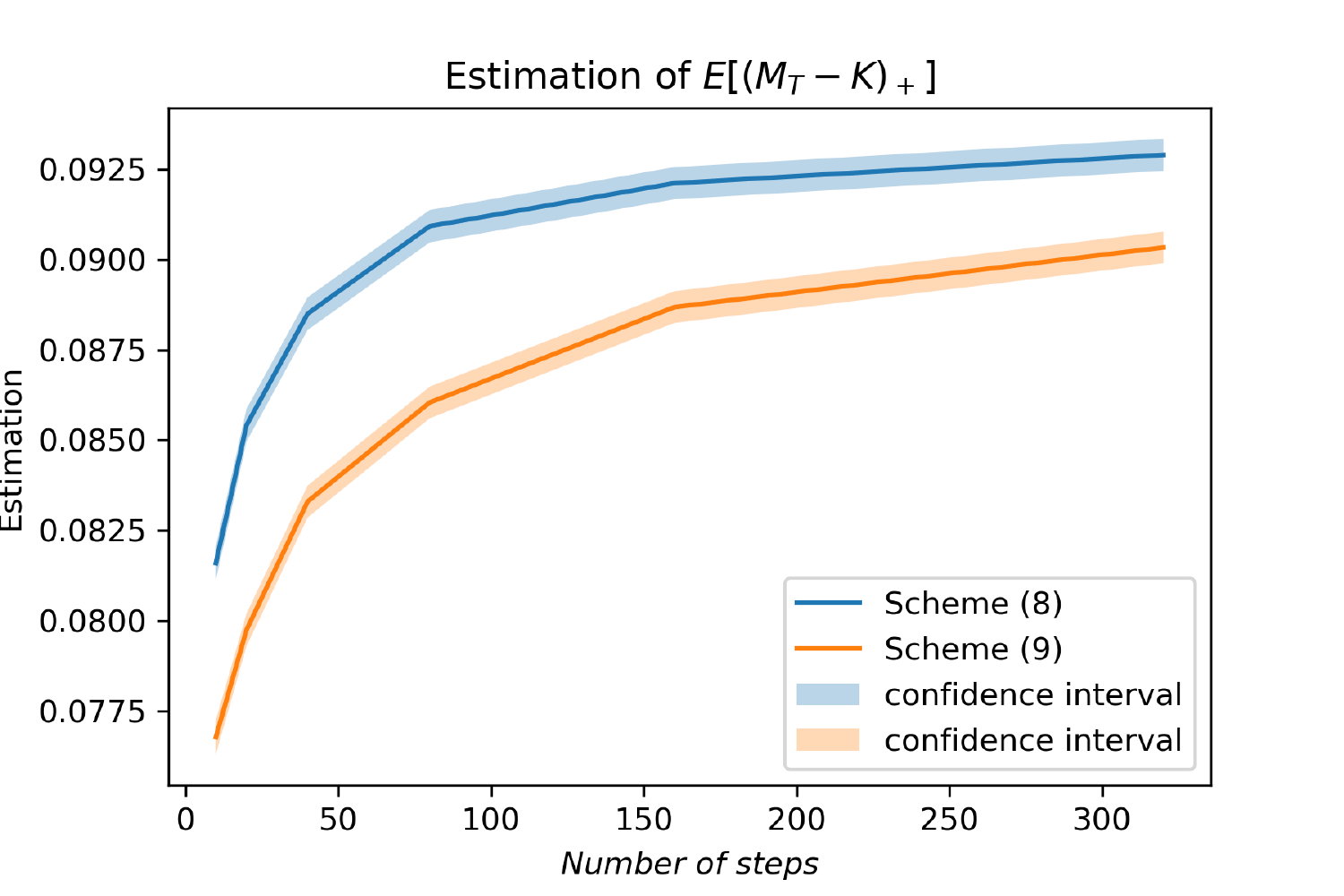}}        
 \caption{\label{fig:Lookback} Estimation of lookback  option price}      
 \end{figure}

\newpage

\subsection{Pricing options on the variance process}

	We next consider the options on the variance process $X$, including the variance option with payoff $X_T$, and the call option on the variance with payoff $(X_T - V_0)_+$.
	Under the Rough Heston model with interest rate $r=0$, the price of the options are given respectively by $\E[X_T]$ and $\E[(X_T - V_0)_+]$.
	For both options, one can compute the reference value by deterministic methods.
	
	\vspace{0.5em}
	
	Indeed, for the variance swap option, one can deduce from the dynamic
	$$
		X_t ~=~ V_0 T+\int_0^t K(t-s) \big( \theta s-\lambda X_s+W_{X_s} \big)~ds
	$$
	that
	$$
		\EE[X_t] ~=~ V_0 t + \int_0^t K(t-s) \big(\theta s-\lambda \EE[X_s] \big)~ds,
	$$
	which is a linear Volterra ODE.
	With the above parameters of the rough Heston model, one obtains $0.028295$ as the reference value for the variance swap option price $\EE[X_{T}]$.
	
	\vspace{0.5em}
	
	For the call option on the variance, one can use the results in \cite{abi2021weak} to compute the characteristic function of $X_T$ by solving the corresponding Riccati ODE,
	and then use the inverse Fourier transform method to compute the value of $\EE[(X_T-V_0)_+]$.
	With the above parameters of the rough Heston model, we obtain 
	$0.013517$ as reference value for $\EE[(X_T-V_0)_+]$.

	\vspace{0.5em}
	
	The discrete-time scheme in \eqref{eq:schemeIntegRough} allows simulating directly the value of $X^n_T$.
	Nevertheless, the discrete-time scheme in \eqref{eq:SchemeRoughHeston} provides simulations of $(V^n_{t_k})_{k=0, \cdots, n}$, 
	we then set $X^n_T := \sum_{k=0}^n V^n_{t_k}{\cdot \Delta t}$ for the Monte-Carlo estimation.
	
	\vspace{0.5em}

	The numerical results are reported  in Figures \ref{fig:VarianceSwap} and \ref{fig:VarianceCall} (see respectively Tables \ref{tab:VarianceSwap} and \ref{tab:VarianceCall} for the data).
	We can observe the convergence of both scheme as the number of time steps $n$ increases. 
	At the same time, for the example on variance swap, the performance of Scheme \eqref{eq:schemeIntegRough}  based on the integrated variance formulation seems slightly better.

\begin{figure}[htb]        
 \center{\includegraphics[width=10cm]  {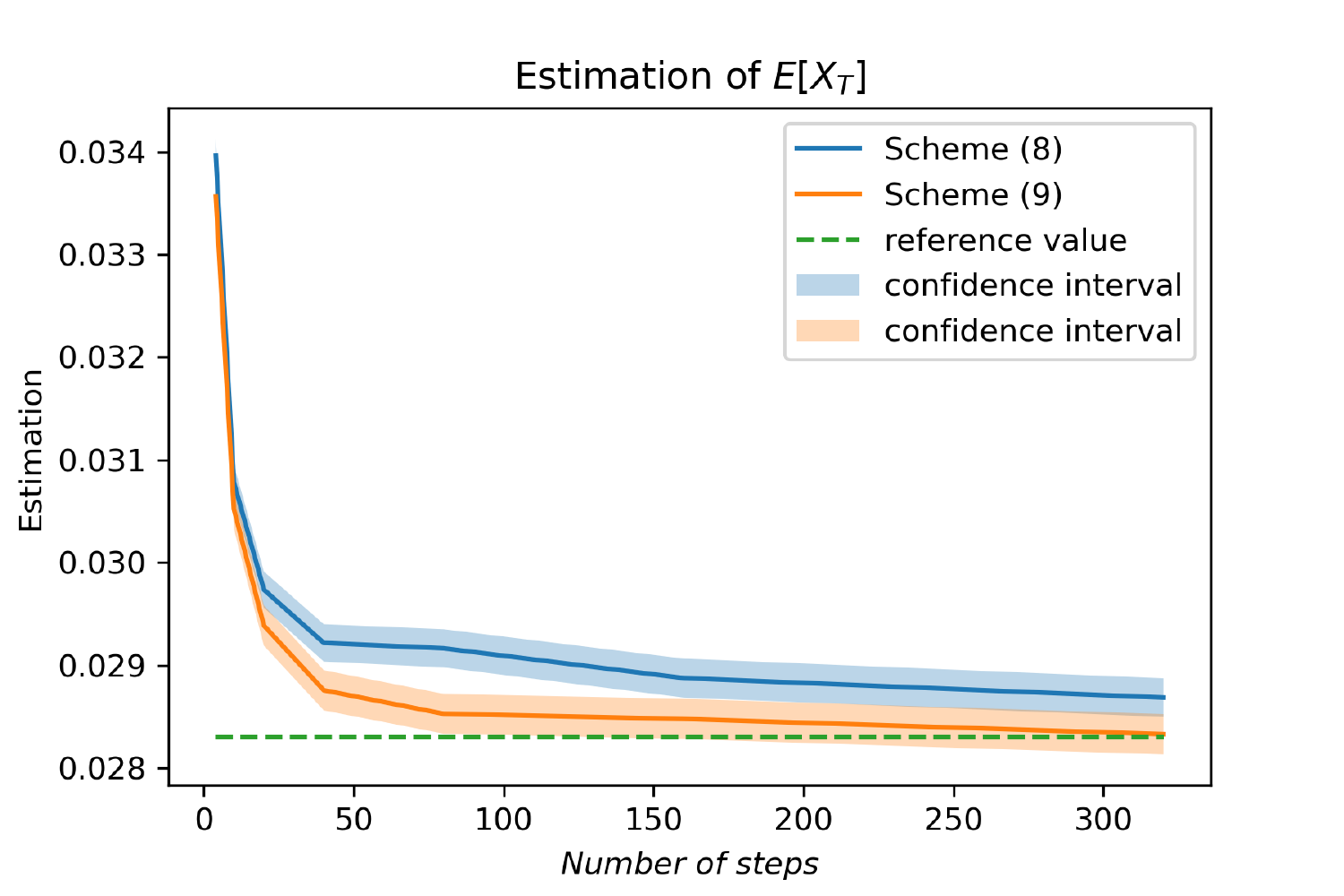}}        
 \caption{\label{fig:VarianceSwap}Estimation of variance swap price}      
 \end{figure}
 
 \begin{figure}[htb]        
 \center{\includegraphics[width=10cm]  {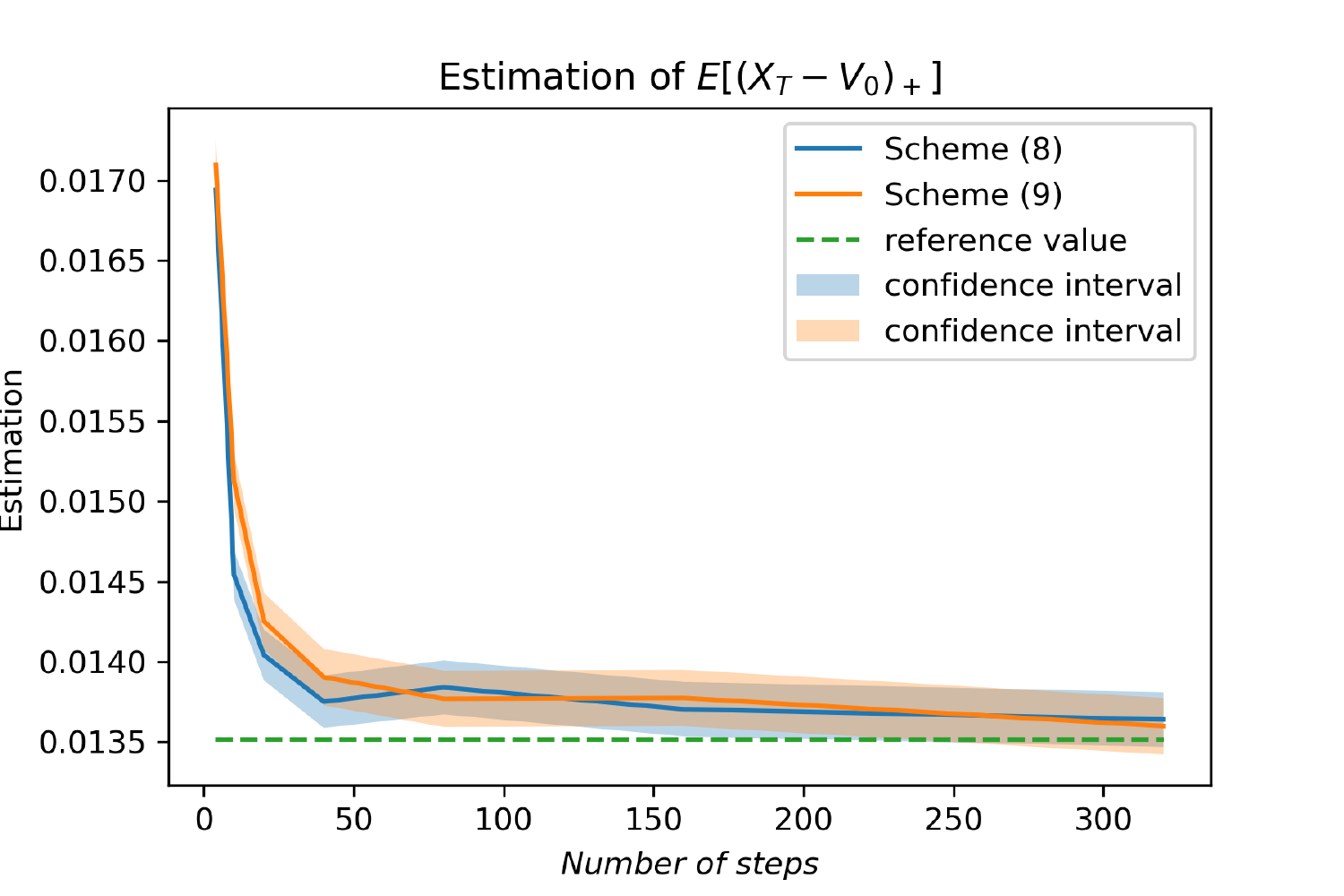}}        
 \caption{\label{fig:VarianceCall} Estimation of variance call price}      
 \end{figure}
 
\newpage

\section{Proofs of Theorems \ref{thm:CvgRoughHeston} and \ref{thm:CvgIntegVariance}}
\label{sec:proofs}

	Recall that, throughout the paper, Assumption \ref{assum:main} holds true.

\subsection{Tightness of solutions to Scheme \eqref{eq:SchemeRoughHeston}}

	Let us first rewrite the system \eqref{eq:RoughHeston} and the corresponding discrete-time scheme \eqref{eq:SchemeRoughHeston} in two-dimensional equations.
	
	\vspace{0.5em}
	
	Let $(S, V)$ be the weak solution of \eqref{eq:RoughHeston}, recall that $Y_t := \log(S_t)$, $v_+ := \max(v, 0)$ for all $v \in \R$.
	We denote
	\begin{equation*}
		\hat{V_t} := \begin{pmatrix}
			Y_t \\ V_t
		\end{pmatrix},~~
		~\hat{K}(t-s)\overset{\triangle }{=}\begin{pmatrix}
			1 & 0 \\
			0 & K(t-s)
		\end{pmatrix},~\hat{W_s}:=\begin{pmatrix}
			W_s^{\bot} \\ W_s
		\end{pmatrix},
	\end{equation*}
	and define functions $b: \R^2 \longrightarrow \R^2$, $\sigma: \R^2 \longrightarrow {\R^{2\times 2}}$ and $a: \R^2 \longrightarrow {\R^{2\times 2}}$ by
	\begin{equation*}
		b\begin{pmatrix}
			y\\ v
		\end{pmatrix} 
		:=\begin{pmatrix}
			-\frac{1}{2} v\\ \theta -\lambda v_+
		\end{pmatrix},
		~~
		\sigma\begin{pmatrix}
			y \\ v
		\end{pmatrix}
		:=  \sqrt{v_+} \begin{pmatrix}
			\sqrt{1-\rho^2} & \rho \\ 0 &\nu
		\end{pmatrix},
		~\mbox{and}~
		a(x) = \sigma \sigma^{\top}(x).
	\end{equation*}
	Notice that $V$ is  $\R_{+}$-valued, then it is easy to check that $\hat{V} = (Y, V)$ satisfies 
	\begin{equation}\label{def:generalRoughHeston}
		\hat{V}_t
		=
		\hat{V}_0 + \int_0^t \hat{K} (t,s)~d\hat{Z}_s,
		~~\mbox{with}~
		\hat{Z_t}:=\int_0^t b(\hat{V_s})~ds+\int_0^t \sigma(\hat{V_s})~d\hat{W_s}. 
	\end{equation}

	Next, for $s,t \in [0, T]$, and $\xbf \in C([0,T],\R^2)$, we define the following new coefficients:
        \begin{equation}\label{def:operators}
	\begin{split}
		b^n(t, \xbf) & := b( \xbf(\ent)),
		~~~
		\sigma^n(t, \xbf ) := \sigma( \xbf(\ent)),
		~~~
		a^n(t, \xbf) :=\big(\sigma^n (\sigma^n)^{\top} \big) (t, \xbf),\\
		K^n(t,s)&:=K(t-\ens),
		~~~
		\hat{K}^{n}(t,s) :=\hat{K}(t-\ens).\\
	\end{split}
	\end{equation}
        	Let $(S^n, V^n)$ be the numerical solution to  Scheme  \eqref{eq:SchemeRoughHeston}.
	Then we can write $\hat{V}^{n} := (Y^n, V^n)$ as solution to the stochastic integral equation
         \begin{equation} \label{eq:schemeRoughHeston_SDE}
             \hat{V}^{n}_t = \hat{V}_0+\int_0^t \hat{K}^{n}(t,s)~d\hat{Z}^n_s,
             ~~\mbox{with}~
             \hat{Z}_s^n = \int_0^s b^n(s,\hat{V}^{n})~ds+\int_0^s \sigma^n(s,\hat{V}^n)~d\hat{W}_s.
         \end{equation}
         
	We next define the infinitesimal generators $\Lc^n$ and $\Lc$ as follows:
	for all $f\in C^2_{b}(\mathbb{R}^2)$ and $(\xbf, \zbf)\in C([0,T],\R^2)\times C([0,T],\R^2)$ and $t \in [0,T]$, let
	\begin{equation}
		\begin{split}\label{def:AnA}
			\Lc^n_t f(\xbf, \zbf) 
			&~=~
			b^n(t, \xbf) \cdot \nabla f( \zbf_t) + \frac{1}{2}  \text{Tr} \big( a^n(t, \xbf) \nabla^2 f(\zbf_t) \big),\\
			\Lc_t f(\xbf, \zbf) 
			&~=~
			b(\xbf_t)  \cdot \nabla f(\zbf_t)+\frac{1}{2}  \text{Tr} \big( a(\xbf_t) \nabla^2 f(\zbf_t) \big).
		\end{split}
	\end{equation}
	It is easy to check that for some constant $C_f > 0$ depending on $f$, one has
	\begin{equation}\label{NormEstimate}
		\big| \Lc^n_t f( \xbf, \zbf) \big| \leq C_f \big(1+| \xbf_{\ent})| \big),
		~~
		\big| \Lc_t f(\xbf, \zbf) \big| \leq C_f \big(1+| \xbf_t| \big),
		~~\mbox{for all}~ \xbf, \zbf \in C([0,T], \R^2).
	\end{equation}

	We next provide a technical lemma.

	\begin{lemma} \label{lem:Gronwall_var}
		Let $(f_n)_{n \ge 1}$ be a sequence of non-negative functions on $[0,T]$ such that $\sup_{n \ge 1} f_n(0) < \infty$.
		Assume that for some constants $q \ge 2$ and $C_0 > 0$, one has
		\begin{align} \label{eq:Gronwall_ineq}
			f_n (t) 
			~\le~&
			C_0 \big(1 + f_n(s) \big) 
			~+~
			C_0 \Big( \int_s^t  \!\! K(t - \eta_n(r))^2 \big( 1+ f_n(\eta_n(r))^{2/q} \big) dr \Big)^{q/2} \\
			&+
			C_0 \Big( \int_0^s \!\! \big|  K(t - \eta_n(r)) -  K(s - \eta_n(r)) \big|^2  \big( 1+ f_n(\eta_n(r))^{2/q} \big) dr \Big)^{q/2},
			~\mbox{for all}~ n \ge 1, ~s \le t. \nonumber
		\end{align}		
		Then 
		$$
			\sup_{n \ge 1}  \sup_{0 \le t \le T} f_n(t)  ~<~ \infty.
		$$
	\end{lemma}
	\begin{proof}
	Let $m \ge 1$ be a fixed integer, and denote $t_k := k T/m$.
	Let us define the nondecreasing function $\bar f_n(t) := \sup_{0 \le s \le t} f_n(s)$.
	We can then apply \eqref{eq:Gronwall_ineq} for $s = t_k$ and $t \in [t_k, t_{k+1}]$ to obtain that
	$$
		f_n(t) ~\le~ (C_0 + C_1 + C_2) +  \big( C_0 + C_1 \big)  \bar f_n(t_k)  + C_2 \bar f_n(t_{k+1}) , ~\mbox{for all}~ t\in [t_k, t_{k+1}],
	$$
	with some positive constants $C_1$, $C_2$ satisfying that, for each $k=0, \cdots, m$, and $t \in [t_k, t_{k+1}]$,
	$$
		C_1 \ge  C_0 \Big( \int_0^{t_k}  \big|  K(t- \eta_n(r)) -  K(t_k - \eta_n(r)) \big|^2 dr \Big)^{q/2},
		~~
		C_2  \ge C_0 \Big( \int_{t_k}^{t} K(t - \eta_n(r))^2 dr \Big)^{q/2}.
	$$
	Without loss of generality, we can assume that $C_0 \ge 1$, so that
	$$
		\bar f_n(t_{k+1}) ~\le~ (C_0 + C_1 + C_2) + \big( C_0 + C_1 \big)  \bar f_n(t_k) +  C_2 \bar f_n(t_{k+1}).
	$$
	Moreover, in view of \eqref{eq:A2}-\eqref{eq:A3}, one can choose a fixed $m \ge 1$ big enough
	so that $C_2 < 1$.
	Then
	$$
		\bar f_n(t_{k+1}) 
		~\le~
		 \frac{C_0 + C_1}{1- C_2}  \bar f_n(t_k) + \frac{C_0 + C_1+C_2}{1- C_2},
		~\mbox{for all}~k =0, \cdots, m-1,
	$$
	and, noticing that $C_0 + C_1 \neq 1- C_2$,
 	it follows that
	$$
		\bar f_n(T)  = \bar f_n(t_m) ~\le~   \Big( \sup_{n\ge 1} f_n(0) \Big) a^m +  \frac{a^m-1}{a-1} b,
		~\mbox{with constants}~
		a := \frac{C_0 + C_1}{1- C_1}, ~ b:= \frac{C_0 + C_1+C_2}{1- C_1}.
	$$
	Notice that the r.h.s. of the above inequality is independent of $n \ge 1$, which is enough to conclude the proof.
	\end{proof}

	\begin{proposition}\label{prop:schemeRoughHest}
		For each $p \ge 2$, there exist a constant $C_p > 0$ such that for all $0\leq s\leq t\leq T$ and $n\geq 1$, 
		\begin{align}
			\EE[|\hat{V}_t^n|^p]&\leq C_p, \label{eq:supEVp}\\
			\EE[|\hat{V}_t^n -\hat{V}_s^n|^p] &\leq C_p (t-s)^{pH}\label{eq:Xn},\\
			\EE[|\hat{Z}_t^n -\hat{Z}_s^n|^p] &\leq C_p (t-s)^{\frac{p}{2}}. \label{eq:Zn}
		\end{align}
	\end{proposition}
	\begin{proof}
	(i) First, one has by Equation \eqref{eq:schemeRoughHeston_SDE} that, for all $0 \le s \le t \le T$,
	\begin{equation} \label{eq:V_st}
	\begin{split}
		\hat{V}^n_t-\hat{V}^n_s=&\int_0^s \Big(\hat{K}(t-\enr)-\hat{K}(s-\enr)\Big)b^n (r,\hat{V}^n)~dr+\int_s^t \hat{K}(t- \enr) \, b^n(r,\hat{V}^n)~dr\\
		&+\int_0^s \Big(\hat{K}(t-\enr)-\hat{K}(s-\enr)\Big)\sigma^n(r,\hat{V}^n)~dW_r+\int_s^t \hat{K}(t-\enr)\sigma^n(r,\hat{V}^n)~dW_r.
	\end{split}
	\end{equation}
	Then using the Burkholder-Davis-Gundy inequality, there is a constant $C > 0$ depending only on $p \ge 2$ such that
	\begin{equation*}
	\begin{split}
		\EE\big[ \big|\hat{V}^n_t-\hat{V}_s^n \big|^p\big]
		&\leq 
		C\, \EE\Big[ \Big|\int_0^s \Big(\hat{K}(t-\enr)-\hat{K}(s-\enr)\Big) b^n (r,\hat{V}^n)~dr \Big|^p \Big] \\
		&~~~+C\, \EE\Big[ \Big|\int_0^s \Big(\hat{K}(t-\enr)-\hat{K}(s-\enr)\Big)^2\sigma^n(r,\hat{V}^n)^2~dr \Big|^{\frac{p}{2}} \Big]\\
		&~~~+C\, \EE\Big[ \Big|\int_s^t \hat{K}(t-r) \, b^n(r,\hat{V}^n)~dr \Big|^p \Big]
		~+C\, \EE\Big[ \Big|\int_s^t \hat{K}(t-\enr)^2\, \sigma^n(r,\hat{V}^n)^2~dr \Big|^\frac{p}{2} \Big].
	\end{split}
	\end{equation*}
	Applying H\"older's inequality and using the linear growth of $b^n$ and $\sigma^n$, 
	there is a positive constant $C > 0$ (independent of $n \ge 1$), such that
	\begin{align*}
		\EE\big[ \big|\hat{V}_t^n - \hat{V}_s^n \big|^p \big]
		~\leq~&
		C \EE \Big[   \Big|\int_0^s \Big(\hat{K}(t-\enr)-\hat{K}(s-\enr)\Big)^2  \big( 1 + (\hat{V}^n_{\enr} )^2 \big)~dr \Big|^{\frac{p}{2}} \Big] \\
		&+
		C \EE\Big[ \Big|\int_s^t \hat{K}(t-\enr)^2\,  \big( 1 + (\hat{V}^n_{\enr} )^2 \big) ~dr \Big|^\frac{p}{2} \Big].
	\end{align*}
	Applying now Minkowski's integral inequality yields
	\begin{equation}\label{eq:incrementsV}
		\begin{split}
			\EE\big[ \big|\hat{V}_t^n - \hat{V}_s^n \big|^p \big]
			~\leq~&
			C  \Big( \int_0^s \Big(\hat{K}(t-\enr)-\hat{K}(s-\enr)\Big)^2  \big( 1 + \EE\big[|\hat{V}^n_{\enr} |^p\big]^{\frac{2}{p}} \big)~dr \Big)^{\frac{p}{2}} \\
			&+
			C \Big(\int_s^t \hat{K}(t-\enr)^2\,  \big( 1 + \EE\big[|\hat{V}^n_{\enr} |^p\big]^{\frac{2}{p}} \big) ~dr \Big)^\frac{p}{2} .
		\end{split}
	\end{equation}
	It follows that there is a positive constant $C_0 > 0$ (independent of $n \ge 1$) such that
	\begin{align*}
		\EE\big[ \big|\hat{V}_t^n \big|^p \big]
		~\leq~&
		C \EE\big[ \big|\hat{V}^n_s \big|^p \big]
		+
		C \EE \big[ \big| \hat{V}^n_t - \hat{V}^n_s \big|^p \big] \\
		~\leq~&
		C_{0} \EE\big[ \big|\hat{V}^n_s \big|^p \big] + C_{0}  \Big( \int_0^s \Big(\hat{K}(t-\enr)-\hat{K}(s-\enr)\Big)^2  \big( 1 + \EE\big[|\hat{V}^n_{\enr} |^p\big]^{\frac{2}{p}} \big)~dr \Big)^{\frac{p}{2}} \\
			&+
			C_{0} \Big(\int_s^t \hat{K}(t-\enr)^2\,  \big( 1 + \EE\big[|\hat{V}^n_{\enr} |^p\big]^{\frac{2}{p}} \big) ~dr \Big)^\frac{p}{2} .
		\end{align*}
	
	Let $f_n(t) :=  \EE \big[ \big|\hat{V}^n_t \big|^p \big]$, one can then apply Lemma \ref{lem:Gronwall_var} to obtain that
	$$
		\sup_{n \ge 1} \sup_{0 \le t \le T} \EE \big[ \big| \hat{V}^n_t \big|^p \big] 
		~=~
		\sup_{n \ge 1} \sup_{0 \le t \le T} f_n(t)
		 ~<~ \infty.
	$$

	\noindent (ii) Plugging the estimate \eqref{eq:supEVp} in \eqref{eq:incrementsV}, one deduces from \eqref{eq:A2}-\eqref{eq:A3}  that, for some constant $C_p > 0$ independent of $n$,
	\begin{equation*}
		\EE\big[ \big|\hat{V}^n_t-\hat{V}_s^n \big|^p\big]
		\leq C_p (t-s)^{pH}.
	\end{equation*}

	\noindent (iii) Notice that $Z^n$ is a semimartingale, with the uniform estimations in \eqref{eq:supEVp}, the estimation in \eqref{eq:Zn} is standard.
\end{proof}

	As a consequence of Proposition \ref{prop:schemeRoughHest}, we immediately obtain the following tightness result.
	\begin{corollary}\label{tightXnZn}
		The sequence $\{(\hat{V}^n,\hat{Z}^n)\}_{n \ge 1}$ is tight in $C([0,T],\R^2) \times C([0,T],\R^2)$.
	\end{corollary}

	\begin{proof}
	By Proposition \eqref{prop:schemeRoughHest},
	we have by choosing $p$ large enough  that
	$$
		\EE[|\hat{V}_t^n-\hat{V}_s^n|^p]\leq C(t-s)^{1+\alpha}
		~~\mbox{and}~
		\EE[|\hat{Z}_t^n-\hat{Z}_s^n|^p]\leq C(t-s)^{1+\alpha'},
	$$
	for some positive constants $C,\, \alpha,\, \alpha'$. 
	Then by Theorem 2.4.11 of \cite{karatzas1998brownian}, one has that $\{\hat{V}^n\}_{n \ge 1}$ and $\{\hat{Z}^n\}_{n \ge 1}$ are tight in $C([0,T],\R^2)$.
	\end{proof}

	Then we identify the equation satisfied by the limit of $\{(\hat{V}^n,\hat{Z}^n)\}_{n \ge 1}$, by considering an appropriate martingale problem.

	\begin{lemma}\label{lem:convbnan}
		Let $A$ be a compact subset of $C([0,T],\R^2)$ under the uniform norm, then
		$$
			\sup_{ \xbf \in A}  \max_{0 \le t \le T} 
				\big( |a^n(t, \xbf) - a( \xbf_t) | +|b^n(t, \xbf) - b( \xbf_t) | \big)
			\longrightarrow 0,
			~\mbox{as}~
			n \longrightarrow \infty.
		$$ 
	\end{lemma}
	\begin{proof}
	Let $A$ be a compact subset of $C([0,T],\R^2)$. 
	Recall from the Arzel\`a-Ascoli Theorem that the following properties hold:
	$$ 
		M := \underset{\underset{0\leq t\leq T}{\xbf \in A}}{\sup} |\xbf_t| <\infty,
		~~~
		\lim_{\delta\rightarrow 0} \sup_{\xbf \in A} m(\xbf,\delta)=0,
		~\mbox{with}~
		m(\xbf,\delta) :=\max\limits_{|s-t|\leq \delta,\, 0\leq s\leq t\leq T}|\xbf_s- \xbf_t|.
	$$
	Since $a$ and $b$ are uniformly continuous on $\{y\in \mathbb{R}^2;~|y|\leq M\}$, one can find, for any $\eps > 0$, an integer $N \ge 1$ such that
	$$
		\sup_{ \xbf \in A}  \max_{0 \le t \le T} 
				\big( |a^n(t, \xbf) - a( \xbf_t) | +|b^n(t, \xbf) - b( \xbf_t) | \big)
		\leq \eps ,~ \mbox{for all}~ n \geq N.
	$$
	This concludes the proof.
	\end{proof}

	\vspace{0.5em}

	The next lemma is a convergence result for the operator $\Lc^n$ defined in \eqref{def:AnA}, which is in line with Lemmas 3.5 and 3.6 of \cite{abi2019weak}.

	\begin{lemma}\label{conv}
		Let $(\xbf^n, \zbf^n)\longrightarrow  (\xbf, \zbf)$ in $C([0,T],\R^2) \times C([0,T],\R^2)$, then for any $f\in C^2_b(\R^2)$, the following uniform convergence holds:
		$$
			\lim_{n\to \infty} \sup_{t\in[0,T]} \Big| \int_0^t \Lc^n_s f(\xbf^n, \zbf^n)~ds - \int_0^t \Lc_sf(\xbf, \zbf)~ds \Big| = 0 .
		$$
	\end{lemma}
	\begin{proof}
	Let us consider the following decomposition:
	\begin{align*}
		&\int_0^t \big|\Lc^n_s f(\xbf^n, \zbf^n)- \Lc_s f(\xbf, \zbf)\big|~ds\\ 
		&\leq 
		\int_0^T \Big|b^n(s, \xbf^n) \cdot \nabla f(\zbf^n_s)-b(\xbf_s) \cdot \nabla f(\zbf_s)\Big|~ds
		+
		\int_0^T \Big| \mathrm{Tr} \big( a^n(s, \xbf^n)\nabla^2 f(\zbf^n_s) -a(\xbf_s)\nabla ^2 f(\zbf_s) \big) \Big|~ds\\
		&=: I_1+I_2.
	\end{align*}
	We will detail the convergence of $I_2$, as the convergence of $I_1$ follows by a simpler argument. We get
	\begin{align*}
		I_2 &\leq \int_0^T \big| \mathrm{Tr} \big( \big(a^n(s, \xbf^n)-a( \xbf^n_s)\big)\nabla^2 f(\zbf^n_s) \big) \big|~ds\\
		&~+\int_0^T \big| \mathrm{Tr} \big( \big(a(\xbf^n_s)-a(\xbf_s)\big)\nabla^2 f(\zbf^n_s) \big) \big|~ds\\
		&~+\int_0^T \big| \mathrm{Tr} \big( a(\xbf_s)~(\nabla^2 f(\zbf^n_s)-\nabla^2 f(\zbf_s)) \big) \big|~ds\\
		&=: I_{21} +I_{22}+ I_{23}.
	\end{align*}
	Since $(\xbf^n)$ converges uniformly, it lies in a compact of $C([0,T],\R^2)$. 
	Hence by Lemma \ref{lem:convbnan}, $a^n(t, \xbf^n)$ converges to $a(\xbf_t)$ uniformly in $t \in [0,T]$. 
	Then with the boundedness of $\nabla^2 f$, this yields the convergence $I_{21} \longrightarrow 0$.
	The convergence of $I_{22}$ to $0$ is a direct consequence of the Lipschitz continuity of $a$, 
	the boundedness of $\nabla f$ and the uniform convergence of $\xbf^n$ towards $\xbf$. 
	The boundedness of $a$, continuity of $\nabla^2 f$ and uniform convergence of $\zbf^n$ then give, 
	by an application of the dominated convergence theorem, the convergence of $I_{23}$ to $0$.
	\end{proof}

	Now we define a martingale problem, which adapts the one from \cite[Def. 3.1]{abi2019weak} to our framework without jumps.
	It extends the usual definition of a martingale problem to take into account the non-Markovian property because of the kernel $K$. 
	As in the classical setting, there is equivalence between being a weak solution of an SDE and being a solution to the associated martingale problem.
	More discussions are provided in Remark \ref{rem:MartingalePb}. 
	We denote by $C_c^2(\mathbb{R}^2)$ the space of twice continuously differentiable functions which are compactly supported.

	\begin{definition} \label{def:MartPb}
		A solution of the (local) martingale problem for $(V_0,K, \Lc)$, where $\Lc$ is the operator given by Equation \eqref{def:AnA}, 
		is a pair $(V,Z)$ of $C([0,T],\R^2)$-valued processes defined on a filtered probability space, such that $V$ is predictable, 
		$Z$ is a continuous semimartingale, the process
		$$
			M_t^f=f(Z_t)-\int_0^t \Lc_s f(V,Z)~ds,~t\in [0,T],
		$$
		is a local martingale for every $f\in C_c^2(\mathbb{R}^2)$, and one has the equality
		\begin{equation} \label{eq:martPb_VZ}
			V_t=V_0 + \int_0^t K(t-s) \, dZ_s, \quad t\in[0,T], ~\mbox{a.s.}
		\end{equation}
	\end{definition}

	\begin{proposition}\label{prop:MartingaleProb}
		Any weak limit of $(\hat{V}^n,\hat{Z}^n)$ is a solution of the local martingale problem for $(\hat{V}_0, \hat{K}, \Lc)$.
	\end{proposition}
	\begin{proof}
	Let $(\Omega^n,\mathcal{F}^n, \F^n = (\mathcal{F}_t^n)_{t \in [0,T]},\mathbb{P}^n)$ be a filtered probability space where $(\hat{V}^n,\hat{Z}^n)$ is defined and $(\hat{V},\hat{Z})$ be a weak limit of $(\hat{V}^n, \hat{Z}^n)$. Fix $f\in C^2_c(\R^2)$ and $m\in \N$. For any $\xbf\in C([0,T],\R^2)$, define
	$$\tau_m(\xbf)=\inf\{ t \geq 0: |\xbf(t)| \geq m \} \wedge T. $$
	Then $\tau_m (\hat{V}^n)$ is an $\F^n$-stopping time, and we have
	$$
		\int_0^{t\wedge\tau_m (\hat{V}^n)} \big| \Lc_s^n f(\hat{V}^n,\hat{Z}^n) \big|~ds 
		~\leq~
		C_f\int_0^{t\wedge\tau_m (\hat{V}^n)} \big(1+|\hat{V}_{\ens}^n| \big)~ds
		~\leq~
		(1+m)\, C_f,~t\in [0,T] . 
	$$
	The Itô formula implies that the following process
	$$
		M_t^n
		=
		f \big( \hat{Z}^n_{t\wedge\tau_m (\hat{V}^n)} \big)
		-
		\int_0^{t\wedge\tau_m(\hat{V}^n)} \Lc_s^n f(\hat{V}^n,\hat{Z}^n)~ds
	$$
	is a uniformly bounded martingale.
	Thus, for any $0\leq t_1 < t_2 ...<t_k \leq s <t \leq T$ and $g_i \in C_b(\mathbb{R}^2 \times\mathbb{R}^2),~i=1,2,...,k$, we have
	\begin{align}\label{expct}
		\EE \Big[ \big( M^n_t-M^n_s \big)\prod_{i=1}^k g_i\big(\hat{V}^n_{t_i},\hat{Z}^n_{t_i} \big) \Big]
		~=~
		0.
	\end{align}
	Then, by Skorokhod's representation Theorem, we may assume that all $(\hat{V}^n,\hat{Z}^n) $ and $(\hat{V},\hat{Z})$ are defined on a common probability space $(\Omega,\mathcal{F},\mathbb{P})$, that $(\hat{V}^n,\hat{Z}^n)\longrightarrow (\hat{V},\hat{Z})$ in $C([0,T],\R^2) \times C([0,T],\R^2)$ almost surely, and that each pair $(\hat{V}^n,\hat{Z}^n)$ has the same law under $\mathbb{P}$ as it did under $\mathbb{P}^n$.

	\vspace{0.5em}

	By Lemma \ref{conv}, one has
	$$
		\int_0^{t\wedge\tau_m(\hat{V}^n)} \Lc^n_s f(\hat{V}^n,\hat{Z}^n)~ds  
		~\longrightarrow~
		\int_0^{t\wedge \tau_m(\hat{V})} \Lc_s f(\hat{V},\hat{Z})~ds,
	$$
	almost surely for any $t \in [0,T]$. 
	Define
	$$
		M_t 
		~:=~
		f(Z_{t\wedge \tau_m(\hat{V})})-\int_0^{t\wedge \tau_m(\hat{V})} \Lc_s f(\hat{V},\hat{Z})~ds,
		~~t\in [0,T].
	$$
	We conclude that $M_t^n \longrightarrow M_t$ almost surely for $t\in [0,T]$. 
	Since $M_t^n $ is bounded uniformly in $n$, 
	we may use the dominated convergence theorem to pass to the limit $n \longrightarrow \infty$ in \eqref{expct} to get
	\begin{align}
		\EE \Big[(M_t-M_s)\prod_{i=1}^k g_i\big(\hat{V}_{t_i}, \hat{Z}_{t_i} \big)\Big]=0.
	\end{align}
	Thus $M$ is a martingale with respect to the filtration given by
	$$
		\mathcal{F}_t=\sigma (\hat{V}_s,\hat{Z}_s: s\leq t),~t\in[0,T].
	$$
	Since $\tau_m (\hat{V})$ is a stopping time for this filtration and the constant $m$ in the definition of $\tau_m (\hat{V})$ was arbitrary, the process $M^f$
	$$M^f_t=f(\hat{Z}_t)-\int_0^t \Lc_s f(\hat{V},\hat{Z})~ds,~ t\in [0,T]$$
 is a local martingale. 
 
 	\vspace{0.5em}
	
	Next, recall that by definition of $\hat{V}^n_t$, we have
	\begin{equation*}
	\begin{split}
		\hat{V}_t^n = \hat{V}_0& + \int_0^t \hat{K}(t-\ens)~d\hat{Z}^n_s.
	\end{split}
	\end{equation*}
	Then it is enough to apply \cite[Theorem 2.2]{kurtz1991weak} (see also \cite{jakubowski1989convergence})
	to conclude that, for each $t \in [0,T]$,
	$$
			\hat{V}_t= \hat{V}_0 + \int_0^t \hat{K}(t-s) \, d \hat{Z}_s, ~\mbox{a.s.}
	$$
	Notice that the process $V$ is continuous in $t$, one then proves \eqref{eq:martPb_VZ}.
	\end{proof}

	\begin{remark} \label{rem:MartingalePb}	
	
	$\mathrm{(i)}$ 
	Our definition of martingale problem in Definition \ref{def:MartPb} differs slightly from \cite[Def. 3.1]{abi2019weak}.
	More precisely, our condition \eqref{eq:martPb_VZ} is on the process $V$, but Condition (3.3) of \cite{abi2019weak} is on the process $\int_0^t V_s \, ds$.
	Nevertheless, as soon as $V$ has continuous paths, the two conditions are in fact equivalent, so that the two definitions are equivalent.
	The formulation based on $\int_0^t V_s \, ds$ has the advantage to obtain the tightness more easily than that of $V$,
	but still we are able to obtain the tightness of $V$ in our context.

	\vspace{0.5em}	
	
	As in the classical setting, it is easy to check that $(\hat{V},\hat{Z})$ is a solution of the local martingale problem (Definition \ref{def:MartPb}) 
	for $(\hat{V}_0,\hat{K}, \Lc)$ if and only if $\hat{V} = (Y, V)$ satisfies
	\begin{equation} \label{eq:RoughHeston_variation}
	\begin{split}
		Y_t &= Y_0+\int_0^t -\frac{1}{2} (V_s)_+~ds+\int_0^t \sqrt{(V_s)_+}~d(\rho W_s+\sqrt{1-\rho^2}W_s^{\bot}),\\
		V_t&=V_0+\int_0^t K(t-s)(\theta-\lambda (V_s)_+)~ds+\int_0^t K(t-s)\nu \sqrt{(V_s)_+}~dW_s.
	\end{split}
	\end{equation}
	The proof is almost the same as Lemma 3.3 of \cite{abi2019weak}, with some simplifications due to the absence of jumps here. It is therefore omitted.
	
	\vspace{0.5em}	
		
	Consequently, to prove the convergence result in Theorem \ref{thm:CvgRoughHeston}, it is enough to prove that Equation \eqref{eq:RoughHeston_variation} has the same unique solution as that of \eqref{eq:RoughHeston}.
	This will be proven in  Proposition \ref{Prop:TwoModifiedModels}.(i).
	\end{remark}

\subsection{Tightness of solutions to Scheme \eqref{eq:schemeIntegRough}}

	Let us first rewrite the discrete scheme \eqref{eq:schemeIntegRough} as an integrated-rough Volterra equation in a continuous-time form:
	$S^n = \exp(Y^n)$, and
	\begin{equation}\label{eq:schemeIRV}
	\begin{cases}
		Y_t^n &\!\!=~ Y_0 +\rho M^n_{t}+\sqrt{1-\rho^2}M^{n,\bot}_t-\frac{1}{2}\overline{X}^n_{t}, \\
		X_t^n &\!\!=~V_0t+\int_0^t K(t-\ens)\Big(\theta \ens-\lambda \overline{X}^n_{\ens}+\nu M^n_{\ens}\Big)~ds, \\
		\overline{X}^n_{t_k}&\!\!=~\max\limits_{0\leq j\leq k}X^n_{t_j}, \\
		\overline{X}^n_t&\!\!=~\overline{X}^n_{t_{k-1}}+(\overline{X}^n_{t_{k}}-\overline{X}^n_{t_{k-1}})\frac{t-t_{k-1}}{t_{k}-t_{k-1}}~,~\mbox{for}~t\in [t_{k-1},t_{k}),\\
		M^n_{t}&\!\!=~ M^n_{t_{k-1}}+\sqrt{\overline{X}^n_{t_{k}}-\overline{X}^n_{t_{k-1}}}\frac{B_t-B_{t_{k-1}}}{\sqrt{t_{k}-t_{k-1}}}~,~\mbox{for}~t\in [t_{k-1},t_{k}), \\
		M^{n,\bot}_{t}&\!\!=~ M^{n,\bot}_{t_{k-1}}+\sqrt{\overline{X}^n_{t_{k}}-\overline{X}^n_{t_{k-1}}}\frac{B^\bot_t-B^\bot_{t_{k-1}}}{\sqrt{t_{k}-t_{k-1}}} ~,~\mbox{for}~t\in [t_{k-1},t_{k}),
 	\end{cases}
	\end{equation}
	where $M^n_{t_0} = M^{n,\bot}_{t_0}=0$, $B$ and $B^\bot$ are independent standard Brownian motions. 
	In view of the above definition, $M^n$ and $M^{n,\bot}$ are continuous martingales with quadratic variation 
	\begin{equation*}
		\langle M^n \rangle_t=\langle M^{n,\bot}\rangle_t=\overline{X}^n_{t_{k-1}}+(\overline{X}^n_{t_{k}}-\overline{X}^n_{t_{k-1}})\frac{t-t_{k-1}}{t_{k}-t_{k-1}} ~,~\mbox{for}~ t\in [t_{k-1},t_{k}),
	\end{equation*}
	which means that for any $t\in [0,T]$,
	\begin{equation}\label{eq:quadVarMn}
		\langle M^n \rangle_t=\langle M^{n,\bot}\rangle_t=\overline{X}^n_{t}.
	\end{equation}

	Recall the following BDG inequality: for some constant $C> 0$, one has 
	\begin{equation} \label{eq:TimeChangedBM}
		\EE \big[ \big| M_t^n-M_s^n \big|^p \big]
		+
		\EE \big[ \big|M^{n,\bot}_t-M^{n,\bot}_s \big|^p \big]
		\leq 
		C\, \EE \big[ \big|\overline{X}^n_t-\overline{X}^n_s \big|^{\frac{p}{2}} \big],
		~~\mbox{for all}~0\leq s\leq t\leq T.
	\end{equation}

	\begin{lemma}\label{convOFz}
		Assume $(\zbf^n)_{n \ge 1} \subseteq C([0,T],\R^2)$ and $\zbf^n\longrightarrow z$ in $C([0,T],\R^2)$. Then the following uniform convergence holds:
		$$
			\lim_{n \to \infty} \sup_{t\in[0,T]} \Big| \int_0^t \hat{K}(t-\ens) \zbf^n_s~ds- \int_0^t \hat{K}(t-s) \zbf_s~ds \Big| = 0 .
		$$
	\end{lemma}
	\begin{proof}
	One has 
	\begin{equation*}
	\begin{split}
		&\Big|\int_0^t \hat{K}(t-s) \zbf_s~ds-\int_0^t \hat{K}(t-\ens) \zbf^n_s~ds\Big| \\
		\leq&
		\Big|\int_0^t \hat{K}(t-s)( \zbf_s- \zbf^n_s)~ds\Big|+\Big|\int_0^t \Big(\hat{K}(t-s)-\hat{K}(t-\ens)\Big) \zbf^n_s ~ds\Big|. \\
	\end{split}
	\end{equation*}
	The term $|\int_0^t \hat{K}(t-s)( \zbf_s- \zbf^n_s)~ds|$ converges to 0 uniformly since $\zbf^n \longrightarrow \zbf$ uniformly 
	and ${|\int_0^t \hat{K}(t-s)~ds|< \infty}$ by the fact that $K \in L^2([0,T])$.
 
	As for the second term, it follows from the boundedness of $(\zbf^n)$ that
	\begin{equation*}
		\Big|\int_0^t \Big(\hat{K}(t-s)-\hat{K}(t-\ens)\Big) \zbf^n_s ~ds\Big|
		 \leq 
		 \sup_{n\in\N} \|\zbf^n\|_\infty \, \int_0^t \Big|\hat{K}(t-s)-\hat{K}(t-\ens)\Big|~ds
		~\longrightarrow~ 0.
	\end{equation*}
	\end{proof}

	Since the processes $\overline{X}^n$ and $\{ \sup_{s\in [0,t]} X^n_{s},\, t\in [0,T]\}$ may differ, we now introduce the auxiliary process
	\begin{equation*}
		\widetilde{X}^n_{t} = \sup_{s\in [0,t]} X^n_{s} .
	\end{equation*}
	We will eventually show that $\overline{X}^n$ and $\widetilde{X}^n$ converge to the same limit, which is $X$.

	\begin{proposition}\label{Prop:Xn}
		For all $p \ge 2$, there exists a constant $C_p > 0$ such that for all $0\leq s\leq t \leq T$ and $n\geq 1$,
		\begin{equation*}
		\begin{split}
			(i)~&\EE \big[\big |X^n_T \big|^p \big] + \EE \big[ \big|\overline{X}_T^n \big|^p \big] 
			~\leq~
			 C_p,\\
			(ii)~& \EE \big[ \big|X^n_t-X^n_s \big|^p \big] + \EE\big[\big|\overline{X}^n_{t}-\overline{X}^n_{s}\big|^p\big]+ \EE \big[ \big|\widetilde{X}^{n}_t -\widetilde{X}^{n}_s\big|^p\big]
            		~\leq~
			C_p\, (t-s)^{pH},\\
			(iii)~& \EE \big[ \big|M^n_{t}-M^n_{s} \big|^{2p} \big]  + \EE \big[ \big|M^{n,\bot}_{t}-M^{n,\bot}_{s} \big|^{2p} \big] 
			~\leq~
			C_p\, (t-s)^{pH}.\\
		\end{split}
		\end{equation*}
	\end{proposition}

	\begin{proof}

	$(i)$ Recall that $K$ is non-increasing, so that 
	$$
		\int_{0}^t |K(t-\ens)|^2 \, ds  ~\le~ \int_0^T |K(s)|^2 ds ~<~ \infty,
		~~\mbox{for all}~n \ge 1
		~\mbox{and}~t \in [0,T].
	$$
	Then by the equation on $X^n$ in \eqref{eq:schemeIRV}, together with Cauchy-Schwarz and H\"older's inequality, 
	one has for some positive constants $C_0$ and $C_1$ (independent of $n \ge 1$)  that
	\begin{align*}
		\big| {X}^n_t \big|^p&\leq C_0 \, |V_0|^p +C_0\, \Big|\int_0^{t} K(t-\ens)\Big(\theta \ens-\lambda \overline{X}^n_{\ens}+\nu M^n_{\ens} \Big)~ds\Big|^p\\
		&\leq C_0\, |V_0|^p+C_1\int_0^t \big( \theta \ens +\lambda |\overline{X}^n_{\ens}|+\nu |M^n_{\ens}| \big)^p~ds.
	\end{align*}
	Using the inequality $\sqrt{x} \leq 1+x$, one obtains that for some constant $C_2 > 0$ (independent of $n$),
	\begin{equation*}
		\EE \Big[\sup_{s \in [0,t]} \big|X_s^n \big|^p \Big]
                ~\leq~
                C_2  + C_2 \int_0^t  \EE \big[ \big|\overline{X}^n_{\ens} \big|^p \big] ~ds.
	\end{equation*}
	Since $\sup_{t\in [0,T]} X^n_{t} = \overline{X}^n_{T}$ and $\overline{X}^n_{\ens} \leq \overline{X}^n_{s}$, 
	it follows by Gr\"onwall's lemma that, for some constant $C_p$ independent of $n \ge 1$,
	$$
		\sup_{0 \le t \le T} \EE \big[ \big| X^n_t \big|^p \big] 
		~\le~
		\EE \big[ \big|\overline{X}_T^n \big|^p \big] 
		~\leq~
		C_p.
	$$


 	$(ii)$ First, we observe that $\overline{X}^n_{t}-\overline{X}^n_{s} \leq \sup\limits_{r\in [s,t]} (X_{r}^n-X^n_{s})$ 
	and $\widetilde{X}^n_{t}-\widetilde{X}^n_{s} \leq \sup\limits_{r\in [s,t]} (X_{r}^n-X^n_{s})$. 
	Hence it suffices to prove that
	\begin{equation}\label{eq:Esup}
		\EE [   \sup_{r\in [s,t]} (X_{r}^n-X^n_{s})^p] \leq C\, (1+|V_0|^p)\, (t-s)^{pH}.
	\end{equation}
	From the Cauchy-Schwarz inequality and \eqref{eq:A2}-\eqref{eq:A3} we get that 
	\begin{align*}
		|X_r^n-X_s^n|^p&= \Big| V_0\, (r-s) + \int_s^r K(r-\enu) \big(\theta \enu -\lambda \overline{X}^n_{\enu}+\nu M^n_{\enu}\big)~du \\
		&\quad\quad +\int_0^s \big(K(r-\enu)-K(s-\enu)\big) \big(\theta \enu -\lambda \overline{X}^n_{\enu}+\nu M^n_{\enu}\big)~du \Big|^p\\
		&\leq C|V_0|^p(r-s)^p+C(r-s)^{pH}\Big|\int_s^r \Big(\theta \enu -\lambda \overline{X}^n_{\enu}+\nu M^n_{\enu}\Big)^2~du\Big|^{\frac{p}{2}}\\
		&~~~+C(r-s)^{pH}\Big|\int_0^s\Big(\theta \enu -\lambda \overline{X}^n_{\enu}+\nu M^n_{\enu}\Big)^2~du\Big|^{\frac{p}{2}}.
	\end{align*}
	Thus taking the supremum on $[s,t]$ and the expectation, and applying again the H\"older inequality ($\frac{p}{2}>1$), one gets
	\begin{align*}
		\EE\Big[ \sup_{r\in [s,t]} |X_r^n-X_s^n|^p \Big] 
		&\leq 
		C |V_0|^p \, (t-s)^p
		+
		C(t-s)^{pH}\int_s^t \EE \Big[ \big| \theta \enu -\lambda \overline{X}^n_{\enu}+\nu M^n_{\enu}\big|^p \Big]~du\\
		&~~~+
		C(t-s)^{pH}\int_0^s \EE \Big[ \big|\theta \enu -\lambda \overline{X}^n_{\enu}+\nu M^n_{\enu}\big|^p \Big]~du.
	\end{align*}
	Hence using \eqref{eq:TimeChangedBM} and point $(i)$, we deduce from the previous inequality that \eqref{eq:Esup} holds.

	\vspace{0.5em}

 	$(iii)$ From \eqref{eq:TimeChangedBM}, we have 
	$$
		\EE \big[ \big|M^n_{t}-M^n_{s} \big|^{2p} \big] \leq C~ \EE\big[ \big|\overline{X}^n_{t}-\overline{X}^n_{s} \big|^p\big].
	$$
	Then by $(ii)$  one obtains the result. 
	Similarly, the same holds for $M^{n,\bot}$.
	\end{proof}

	\begin{corollary}\label{WeakConvergX}
		By passing to a subsequence, one has $(X^n,\overline{X}^n,M^n,M^{n,\bot})\Longrightarrow (X,\overline{X},M,M^{\bot})$ for some stochastic processes $X,~M,~M^{\bot}$ in $C([0,T],\R)$, where $\overline{X}_t=\sup_{r\in [0,t]}X_r$ and $M,~M^\bot$ are two martingales in $C([0,T],\R)$ with quadratic variation $\langle M\rangle=\langle M^{\bot}\rangle=\overline{X}$ and covariation $\langle M,M^{\bot}\rangle=0$.
	\end{corollary}

	\begin{proof}
	In view of Proposition \ref{Prop:Xn} and Theorem 2.4.11 of \cite{karatzas1998brownian}, the sequences of processes $\{X^n\}_{n\in \N}$, $\{\widetilde{X}^{n}\}_{n\in \N}$, $\{\overline{X}^{n}\}_{n\in \N}$, $\{M^n\}_{n\in \N}$ and $\{M^{n,\bot}\}_{n\in \N}$ are tight in $C([0,T],\R)$. Up to passing to a subsequence (without loss of generality, we do not rename the processes), one has $(X^n,\widetilde{X}^{n},M^n,M^{n,\bot}) \Longrightarrow (X,Y, M,M^{\bot}) $ for some processes $X,Y,M,M^{\bot}$  in $C([0,T],\R)$.
    
	\vspace{0.5em}
	
	\noindent For any $s_1,~s_2,...,s_k \in [0,T]$, the mapping        
	$$
		C([0,T],\R)  \ni \xbf 
		~\longmapsto~
		\Big(\sup\limits_{r\in [0,s_1]} \xbf_r, \sup\limits_{r\in [0,s_2]} \xbf_r,\dots,\sup\limits_{r\in [0,s_k]} \xbf_r \Big) \in \R^k
	$$
	is continuous. So by the continuous mapping theorem and the convergence in law of $X^n$ to $X$, there is    
	$$
		\big( \widetilde{X}^{n}_{s_1},\widetilde{X}^{n}_{s_2}, \dots, \widetilde{X}^{n}_{s_k} \big)
		\Longrightarrow 
		\big( \overline{X}_{s_1},\overline{X}_{s_2},\dots ,\overline{X}_{s_k} \big).
	$$
	Moreover the convergence in law of $\widetilde{X}^{n}$ to $Y$ implies that
	$$
		\big( \widetilde{X}^{n}_{s_1},\widetilde{X}^{n}_{s_2},\dots,\widetilde{X}^{n}_{s_k} \big)
		\Longrightarrow 
		\big( Y_{s_1},Y_{s_2},...,Y_{s_k} \big).
	$$
	Thus the finite dimensional distributions of $\overline{X}$ and $Y$ are identical, and since the class of finite dimensional distributional sets is a separating class in $C([0,T],\R)$, we obtain $Y=\overline{X}$. 
    
	\vspace{0.5em}

	We now aim at proving that $\overline{X}^n \Longrightarrow \overline{X}$. 
	In view of the convergence of $\widetilde{X}^n$ to $\overline{X}$, it suffices to prove that $\sup_{t\in [0,T]}|\widetilde{X}^{n}_{t}-\overline{X}^n_{t}|\Longrightarrow 0$.
 
 	For $t\in [0,T]$, one has 
	$$\widetilde{X}_t^{n}-\overline{X}^n_t  \leq \widetilde{X}^{n}_t-\overline{X}^n_{\ent} \leq \sup\limits_{s\in [0,t]}(X^n_s -X^n_{\ens}).$$
	From Proposition \ref{Prop:Xn} and the Markov inequality, one gets that
	$$
		\sup\limits_{t\in [0,T]}|\widetilde{X}^{n}_t-\overline{X}^n_t| \Longrightarrow 0.
	$$
	Hence $\overline{X}^n \Longrightarrow \overline{X}$.

	\vspace{0.5em}
    
	\noindent From $(M^n,M^{n,\bot},\overline{X}^n)\Longrightarrow (M,M^{\bot},\overline{X})$, we have 
	$$(M^n)^2-\overline{X}^n\Longrightarrow M^2-\overline{X},$$
	$$(M^{n,\bot})^2-\overline{X}^n\Longrightarrow (M^{\bot})^2-\overline{X},$$
	$$M^n M^{n,\bot} \Longrightarrow MM^{\bot}.$$
	By Proposition IX.1.17 in \cite{jacod2013limit}, the processes $M,~M^{\bot},~M^2-\overline{X},~(M^{\bot})^2-\overline{X},~MM^{\bot}$ are all continuous local martingales w.r.t. the filtration generated by $(M, M^{\bot}, \overline{X}, X)$.
	Therefore, $\langle M\rangle= \langle M^{\bot}\rangle=\overline{X},~\langle M,M^{\bot}\rangle=0$. 
	Moreover, by the uniform moment estimates on $(M^n, M^{n,\bot})$ (see Proposition \ref{Prop:Xn}), it follows that they are continuous martingales.
\end{proof}

	As an immediate consequence of Corollary \ref{WeakConvergX}, we get that up to taking a subsequence, the following convergence holds:
	$$
		(S^n, Y^n) \Longrightarrow (S, Y),
		~\mbox{with}~
		S := \exp(Y)
		~\mbox{and}~
		Y := \rho M+\sqrt{1-\rho^2}M^{\bot}-\frac{1}{2}\overline{X}.
	$$
   
   	We next prove that the limit processes $(X, \overline X, S, Y, M, M^{\bot})$ satisfy
	\begin{equation} \label{eq:IntegVarHeston_variation}
	\begin{cases}
		X_t=V_0 t+\int_0^t K(t,s)(\theta s-\lambda \overline{X}_s+\nu M_s)~ds, \\
		\overline X_t = \max_{0 \le s \le t} X_s, \\
		S_t=S_0 \exp(\rho M_t+\sqrt{1-\rho^2}M^{\bot}_t-\frac{1}{2}\overline{X}_t),\\
		M,\, M^\bot \mbox{ are continuous martingales with } \langle M\rangle = \langle M^\bot\rangle = \overline{X}, ~\langle M, M^\bot\rangle=0 ~\mbox{and}~  M_0=M^\bot_0 =0.
	\end{cases}
	\end{equation}

	\begin{proposition}\label{prop:ConvergXn}
		Any weak limit $(X, \overline X, S, M, M^{\bot})$ of $(X^n, \overline X^n, S^n, M^n, M^{n,\bot})_{n \ge 1}$ in Corollary \ref{WeakConvergX} satisfies Equation \eqref{eq:IntegVarHeston_variation}.
	\end{proposition}
	\begin{proof}
	Let $(\Omega^n,\mathcal{F}^n, \F^n = (\mathcal{F}^n_t)_{t \in [0,T]}),\mathbb{P}^n)$ be a filtered probability space in which $(X^n,\overline{X}^n,M^n,S^n)$ is defined. 
	Corollary \ref{WeakConvergX} gives the existence of a subsequence $(n_{k})$ such that  $(X^{n_{k}},\overline{X}^{n_{k}},M^{n_{k}}, S^{n_{k}})$ converges, and without loss of generality we simply consider that 
	$(X^n,\overline{X}^n,M^n, S^n) \Longrightarrow (X,\overline{X},M, S)$, for some $X$ in $C([0,T],\R)$ and $M$ a martingale with quadratic variation $\overline{X}$ in a filtered probability space $(\Om, \Fc, \F= (\Fc_t)_{t \in [0,T]}, \P)$. 
	By the Skorokhod representation theorem, one can assume that $(\Om^n, \Fc^n, \P^n) = (\Om, \Fc, \P)$ (but the filtrations $\F^n$ and $\F$ are different a priori)
	such that $({X}^n,{\overline{X}^n}, {M}^n,{S}^n) \longrightarrow ({X},{\overline{X}},{M},{S})$ almost surely in $C([0,T],\R)^4$. 
	From Lemma \ref{convOFz}, one has 
	$$
            	\int_0^t K(t-\ens)(\theta \ens-\lambda {\overline{X}^n}_{\ens}+\nu {M}^n_{\ens})~ds 
		\longrightarrow 
		\int_0^t K(t-s)(\theta s-\lambda {\overline{X}}_s+\nu {M}_s)~ds            
	$$ 
	uniformly in $t\in [0,T]$, almost surely.
	Therefore, $(X,S,M, M^{\bot})$ solves \eqref{eq:IntegVarHeston_variation}.
	\end{proof}

\subsection{Equivalence and uniqueness of weak solutions}
\label{subsec:equivalence}

	We will adapt the ideas in the proof of \cite[Theorem A.1]{abi2019markovian}, 
	in order to prove that the process $V$ in a (weak) solution of \eqref{eq:RoughHeston_variation} is non-negative,
	and that the process $X$ in a (weak) solution of \eqref{eq:IntegVarHeston_variation} is  non-decreasing.
	Consequently, Equation \eqref{eq:RoughHeston_variation} shares the same unique weak solution as \eqref{eq:RoughHeston},
	and Equation \eqref{eq:IntegVarHeston_variation} shares the same unique weak solution as \eqref{eq:IntegVarHeston}.
	
	\vspace{0.5em}
	
	Recall that the resolvent $L$ of the first kind of the kernel function $K$ is defined in \eqref{eq:def_resolvent},
	together with the definition of the convolution:
	$$
		(L^**K)(t)
		=
		(K*  L^*) (t) 
		:= \int_{[0,t]} K(t-s)  L^*(ds),
		~~\mbox{for a finite measure}~ L^* ~\mbox{on}~[0,T].
	$$
	Let us fix $K(t) := K(T)$ for $t \ge T$,
	and denote $\Delta_h K( \cdot)=K(h+ \cdot)$.

	\begin{proposition}\label{Prop:TwoModifiedModels}
		\noindent(i) Let $V$ be a (weak) solution to the following equation
		\begin{equation}\label{eq:ModifiedRoughHeston}
		\begin{split}
			V_t
			~=~
			V_0
			+
			\int_0^t K(t-s) \big( \theta -\lambda (V_s)_+ \big)~ds
			+
			\int_0^t K(t-s)\nu \sqrt{(V_s)_+}~dW_s,
		\end{split}
		\end{equation}
		where $x_+=max(x,0)$.
		Then $V$ is non-negative on $[0,T]$.

		\vspace{0.5em}

		\noindent(ii) Let $X$ be a (weak) solution to the following equation
		\begin{equation}\label{eq:modifiedHyperModel}
		\begin{split}
			X_t
			~=~
			V_0 t
			+
			\int_0^t K(t-s) \big( \theta s-\lambda \overline{X}_s \big)~ds
			+
			\int_0^t K(t-s)\nu M_s~ds,
		\end{split}
		\end{equation}
		where $\overline{X_t}=\sup\limits_{s\in [0,t]}X_s$, $M_s$ is a continuous martingale with quadratic variation $\langle M\rangle _s=\overline{X}_s$ and $M_0=0$.
		Then $X$ is non-decreasing on $[0,T]$.
	\end{proposition}

	\begin{proof}
	(i)
	Define the stopping time $\tau_n :=\inf\{t:V_t < -\frac{1}{n}\}$. 
	Assume that the set $\{ \tau_n < T \}$ is nonempty,
	then for every fixed $\omega \in \{\tau_n < T\}$, there exists $h > 0$ such that
	\begin{equation} \label{eq:V_tauh_le_1n}
		V_{\tau_n +h} < -\frac1n,
		~~\mbox{and}~
		V_s \leq 0
		~\mbox{for all}~
		s \in [\tau_n, \tau_n+h].
	\end{equation}
	Denote 
	$$
		Z_s := \int_0^s b((V_s)_+)~ds+\int_0^t \sigma((V_s)_+)~dW_s,
		~\mbox{with}~
		b(v) := \theta - \lambda v_+,
		~~
		\sigma(x) = \nu \sqrt{x_+}.
	$$
	It follows that
	\begin{equation*} 
		V_{t+h}
		~=~
		V_0+(K*dZ)_{t+h}
		~=~
		V_0+(\Delta_h K*dZ)_t+Y_t,
	\end{equation*}    
	where 
	$$
		Y_t=\int_t^{t+h}K(t+h-s)~dZ_s,~t\geq 0.
	$$
	By (3.7) of \cite{abi2019affine}, one has for any stopping time $\tau <\infty$, 
	$$
		Y_{\tau }=\int_{\tau }^{\tau+h} K(\tau +h-s)~dZ_s,~\mbox{a.s.}
	$$
	Then one has
	$$
		V_{\tau_n +h} ~=~ V_0+(\Delta_h K*dZ)_{\tau_n}+\int_{\tau_n}^{\tau_n+h}K(\tau_n+h-s)~dZ_s.
	$$
	Recall also from \cite[Lemma B.2]{abi2019markovian} that for any $F\in L^1([0,T])$ such that $F*L$ is right-continuous and of bounded variation, one has 
	$$
		F(t) = (F*L)(0) K(t)
		+
		\big( K* (d(F*L)) \big)(t),
		~\mbox{for a.e.}~t \in [0,T].
	$$
	Then for $F = \Delta_h K$ and $L$ the resolvent of the first kind of $K$, $\Delta_h K *L$ is right-continuous and of bounded variation (see \cite[Remark B.3.]{abi2019markovian}). Thus one has
	\begin{equation*}
	\begin{split}
		V_{\tau_n +h}  
		~=~&
		V_0+(\Delta_h K*L)(0) \big( V_{\tau_n}-V_0 \big) + \big(d(\Delta_h K*L)*(V-V_0)\big)_{\tau_n}+\int_{\tau_n}^{\tau_n+h}K(\tau_n+h-s)~dZ_{s}\\
		~=~&
		V_0 ~+~ V_{\tau_n} (\Delta_h K*L)(0) 
		~+ 
		\int_0^{\tau_n} V_{\tau_n -s} d \big(\Delta_h K *L\big) (s) 
		~-~
		V_0 ~\big( \Delta_h K *L \big)(\tau_n) \\
		&~+ \int_{\tau_n}^{\tau_n+h}K(\tau_n+h-s)~dZ_{s}.
	\end{split}
	\end{equation*}
	Notice that $V_{\tau_n} = -\frac1n$, 
	$V_{\tau_n-s} \ge -\frac1n$ for $s \in [0, \tau_n]$,
	and $(\Delta_hK*L)$ is non-negative and non-decreasing 
	(see \cite[Remark B.3.]{abi2019markovian}). It follows that
	\begin{align*}
		V_{\tau_n +h} 
		~\ge~&
		V_0 ~-~ \frac1n (\Delta_hK*L)(0)  ~-~ \frac1n (\Delta_hK*L)(\tau_n) ~+~ \frac1n (\Delta_hK*L)(0) \\
		&~-~ V_0 (\Delta_h K*L)(\tau_n) 
		~+~
		\int_{\tau_n}^{\tau_n+h} \!\!\! K(\tau_n+h-s)~dZ_{s}.
	\end{align*}
	Further, notice that for all $t \in [0,T]$,
	\begin{align*}
		(\Delta_h K * L)(t) 
		~=~
		\int_0^t K(t+h-s)L(ds)
		~\le~
		\int_0^{t+h}K(t+h-s)L(ds)
		~=~
		(\Delta_h K*L)(t+h)
		~=~
		1,
	\end{align*}
	since $K$ and $L$ are non-negative.
	Then
	$$
		V_{\tau_n +h} 
		~\ge~
		- \frac1n (\Delta_hK*L)(\tau_n)
		+ 
		\int_{\tau_n}^{\tau_n+h}K(\tau_n+h-s)~dZ_{s}
		~\ge~
		-\frac1n
		+\int_{\tau_n}^{\tau_n+h}K(\tau_n+h-s)~dZ_{s}.
	$$
	Finally, since $b((V_s)_+)\geq 0$, $\sigma((V_s)_+)= 0$ for all $s\in [\tau_n,\tau_n+h]$,
	it follows that 
	$$ 
		\int_{\tau_n}^{\tau_n+h}K(\tau_n+h-s)~dZ_s
		~=~
		\int_{\tau_n}^{\tau_n+h} K(\tau_n+h-s)
			\Big( b((V_s)_+)~ds
			+
			\sigma( (V_s)_+)~dW_s \Big)
		~\geq~
		0,
	$$
	and hence $V_{\tau_n +h} \ge - \frac{1}{n}$, which is a contradiction to the fact that $V_{\tau_n + h} < -\frac1n$ in \eqref{eq:V_tauh_le_1n}.
	Therefore, $\P[ \tau_n  < T ] = 0$.
	Since $n\ge 1$ is arbitrary, it follows that $V_t \geq 0$ for any $t\in[0,T]$, a.s.
	        
        \vspace{0.5em}
        
        \noindent (ii) For equation \eqref{eq:modifiedHyperModel}, denote $Z_s := \theta s-\lambda \overline{X}_s+\nu M_s$, 
	where $M$ is a continuous martingale with quadratic variation $\langle M \rangle_s=\overline{X}_s,~M_0=0$. 
	We follow the proof of \cite[Lemma 2.1]{abi2021weak} and apply the stochastic Fubini theorem (see e.g. \cite[p.175]{RevuzYor}) to obtain that
	\begin{equation*}
	\begin{split}
		X_t 
		&~=~
		V_0 t+\int_0^t  K(t-s)Z_s~ds
		~=~
		V_0 t+\int_0^t K(s)\Big(\int_0^{t-s}~dZ_r\Big)~ds\\
		&~=~
		V_0 t+\int_0^t\Big(\int_0^{t-r}K(s)~ds\Big) ~dZ_r
		~=~
		V_0 t+ \int_0^t\Big(\int_0^t K(s-r)\textbf{1}_{r\leq s}~ds\Big) ~dZ_r\\
		&~=~
		V_0 t+\int_0^t\Big(\int_0^sK(s-r)~dZ_r\Big)~ds.
            \end{split}
	\end{equation*}
	Then $X$ is absolutely continuous in $t$ so that both $\frac{d {X}_t}{dt}$ and $\frac{d\overline{X}_t}{dt}$ are well defined, and
	$$
		X_t
		~=~
		\int_0^t \Big(V_0+ \int_0^s K(s-r)b\big(\frac{d\overline{X}_r}{dr}\big)~dr+\int_0^s K(s-r)\nu \sqrt{\frac{d\overline{X}_r}{dr}}~d \hat W_r\Big)~ds,
	$$
	with $b(v) := \theta - \lambda v_+$ and for some Brownian motion $\hat W$.
	Let us define $V_t :=   \frac{d {X}_t}{dt}$, then $V$ satisfies
	$$
		V_t
		~=~
		V_0+ \int_0^t K(t-s)b\big(\frac{d\overline{X}_s}{ds}\big)~ds+\int_0^t K(t-s) \nu \sqrt{\frac{d\overline{X}_s}{ds}}~d \hat W_s.
	$$ 
	Notice that $V_t < 0$ implies that $\frac{d\overline{X}_t}{dt} = 0$,
	then we can apply the same arguments as in (i) to deduce that $ V_t\geq 0$ for any $t\in [0,T]$.
        Since $X_t=\int_0^t  V_s~ds$, this proves that all solutions to equation \eqref{eq:modifiedHyperModel} are non-decreasing over [0,T].
	\end{proof}

\subsection{Proofs of Theorems \ref{thm:CvgRoughHeston} and \ref{thm:CvgIntegVariance}}
	
	In view of Proposition \ref{Prop:TwoModifiedModels}, any weak solution $V$ to \eqref{eq:ModifiedRoughHeston} is non-negative.
	Therefore, a weak solution to \eqref{eq:RoughHeston_variation} is also a weak solution to \eqref{eq:RoughHeston}, which is unique (see Remark \ref{rem:uniqueness_RoughHeston}).
	One then proves Theorem \ref{thm:CvgRoughHeston} by Proposition \ref{prop:MartingaleProb} and Remark \ref{rem:MartingalePb}.
	
	\vspace{0.5em}

	Again by Proposition \ref{Prop:TwoModifiedModels}, any weak solution $X$ to \eqref{eq:modifiedHyperModel} is non-decreasing.
	Therefore, a weak solution to \eqref{eq:IntegVarHeston_variation} is also a weak solution to \eqref{eq:IntegVarHeston}, which is unique (see Remark \ref{eq:uniqueness_VarHeston}).
	One can then prove Theorem \ref{thm:CvgIntegVariance} by Proposition \ref{prop:ConvergXn}.
	\qed

		

\appendix

\section{Tables of simulation results}

\begin{table}[!htbp]
    \centering
    \begin{tabular}{|c||c|c|c||c|c|c|}
    \hline
         & Mean Value &Stat. Error & Comp. Time&Mean Value&Stat. Error&Comp. Time 
         \\ \hline

        Ref.  & 0.056832&-&-&-&-&-
        \\ \hline

         n=4 &0.059642 & 0.000245 & 7.856192&0.065483 & 0.000279 & 7.050599
        \\ \hline

          n=10 &0.058905 & 0.000238 & 17.419916&0.059996 & 0.000244 & 16.307883
        \\ \hline

      n=20 &0.058630 & 0.000234 & 35.061742&0.058635 & 0.000234 & 31.698260
        \\ \hline

      n=40 &0.058344 & 0.000232 & 70.921679&0.057363 & 0.000228 & 60.035371
        \\ \hline

     n=80 &0.058280 & 0.000232 & 129.678460&0.056967 & 0.000225 & 123.884151
        \\ \hline

     n=160&0.057965 & 0.000230 & 266.368216&0.056905 & 0.000225 & 247.920679
        \\ \hline
    n=320&0.057858 & 0.000229 & 582.703205&0.056897 & 0.000225 & 482.391579
        \\ \hline
        
	\end{tabular}
	\caption{European call $(S_T - K)_+$ option price estimation with Scheme \eqref{eq:SchemeRoughHeston}  (left) and Scheme \eqref{eq:schemeIntegRough} (right). The computation time is in second.} 
	\label{tab:EuropCall}
\end{table}

\begin{table}[!htbp]
    \centering
    \begin{tabular}{|c||c|c|c||c|c|c|}
    \hline
         & Mean Value &Stat. Error & Comp. Time & Mean Value & Stat.Error & Comp. Time
         \\ \hline

        Ref.  & -&-&-&-&-&-
        \\ \hline

       n=10 &0.036363 & 0.000145 & 18.126627&0.038373 & 0.000155 & 16.334577
        \\ \hline

        n=20 &0.034653 & 0.000136 & 34.257258&0.034992 & 0.000138 & 31.231823
        \\ \hline

         n=40 &0.033647 & 0.000132 & 68.872156&0.033387 & 0.000130 & 61.623905
        \\ \hline

      n=80 &0.033266 & 0.000130 & 144.862723&0.032474 & 0.000126 & 121.671897
        \\ \hline

         n=160&0.032915 & 0.000128 & 292.438511&0.032230 & 0.000125 & 239.012198

        \\ \hline
       n=320 &0.032479 & 0.000127 & 563.534848& 0.031907 & 0.000124 & 450.094275
        \\ \hline
        
	\end{tabular}
	\caption{Asian option $(A_T - K)_+$ price estimation with Scheme \eqref{eq:SchemeRoughHeston}  (left) and Scheme \eqref{eq:schemeIntegRough} (right). The computation time is in second.}
	\label{tab:Asian}
\end{table}

\begin{table}[!htbp]
    \centering
    \begin{tabular}{|c||c|c|c||c|c|c|}
    \hline
         & Mean Value &Stat. Error & Comp. Time & Mean Value &Stat. Error &Comp. Time
         \\ \hline

        Ref.  & -&-&-&-&-&-
        \\ \hline

       n=10 &0.081591 & 0.000234 & 18.394217&0.076769 & 0.000238 & 17.260538
        \\ \hline

       n=20 &0.085409 & 0.000228 & 33.485885&0.079742 & 0.000226 & 33.581574
        \\ \hline

   n=40 &0.088502 & 0.000226 & 66.453965& 0.083297 & 0.000222 & 64.918632
        \\ \hline

      n=80&0.090927 & 0.000225 & 130.933455&0.086043 & 0.000219 & 124.431435
        \\ \hline

         n=160&0.092118 & 0.000223 & 290.638306&0.088685 & 0.000219 & 246.881069

        \\ \hline
        n=320&0.092904 & 0.000222 & 558.946791&0.090338 & 0.000219 & 492.821703
        \\ \hline
        
	\end{tabular}
	\caption{Lookback option $(M_T - K)_+$ price estimation with Scheme \eqref{eq:SchemeRoughHeston}  (left) and Scheme \eqref{eq:schemeIntegRough} (right). The computation time is in second.}
	\label{tab:Lookback}
\end{table}

\begin{table}[!htbp]
    \centering
    \begin{tabular}{|c||c|c|c||c|c|c|}
    \hline
         & Mean Value &Stat. Error & Comp. Time & Mean Value &Stat. Error &Comp. Time
         \\ \hline

        Ref.  & 0.028295&-&-&-&-&-
        \\ \hline

        n=4&0.033967 & 0.000080 & 8.389909&0.033565 & 0.000090 & 8.291969
\\ \hline
       n=10 &0.030781 & 0.000085 & 17.706663&0.030522 & 0.000098 & 17.645600
        \\ \hline

       n=20 &0.029736 & 0.000088 & 33.061506&0.029387 & 0.000098 & 31.253372
        \\ \hline

   n=40 &0.029218 & 0.000091 & 69.696746& 0.028756 & 0.000098 & 59.500006
        \\ \hline

      n=80&0.029165 & 0.000093 & 132.443006&0.028527 & 0.000097 & 123.055661
        \\ \hline

         n=160&0.028875 & 0.000095 & 267.132785&0.028477 & 0.000098 & 248.295026

        \\ \hline
        n=320&0.028685 & 0.000094 & 575.724998&0.028328 & 0.000098 & 486.069715
        \\ \hline
        
	\end{tabular}
	\caption{Estimation of $\EE[X_T]$  with Scheme \eqref{eq:SchemeRoughHeston}  (left) and Scheme \eqref{eq:schemeIntegRough} (right). The computation time is in second.}
	\label{tab:VarianceSwap}
\end{table}

\begin{table}[!htbp]
    \centering
    \begin{tabular}{|c||c|c|c||c|c|c|}
    \hline
         & Mean Value &Stat. Error & Comp. Time & Mean Value &Stat. Error &Comp. Time
         \\ \hline

        Ref.  & 0.013517&-&-&-&-&-
        \\ \hline

        n=4&0.016940 & 0.000074 & 8.425556&0.017096 & 0.000081 & 7.537352
\\ \hline
       n=10 &0.014542 & 0.000076 & 18.414490&0.015131 & 0.000088 & 15.891882
        \\ \hline

       n=20 &0.014043 & 0.000079 & 36.577253&0.014254 & 0.000088 & 30.900766
        \\ \hline

   n=40 &0.013752 & 0.000082 & 70.948186& 0.013905 & 0.000088 & 59.677628
        \\ \hline

      n=80&0.013841 & 0.000084 & 144.526883&0.013770 & 0.000088 & 118.789620
        \\ \hline

         n=160&0.013705 & 0.000085 & 273.132594&0.013775 & 0.000088 & 230.445799

        \\ \hline
        n=320&0.013641 & 0.000085 & 585.570704&0.013600 & 0.000088 & 471.164119
        \\ \hline
        
	\end{tabular}
	\caption{Estimation of $\EE[((X_T-V_0)_+]$  with Scheme \eqref{eq:SchemeRoughHeston}  (left) and Scheme \eqref{eq:schemeIntegRough} (right). The computation time is in second.}
	\label{tab:VarianceCall}
\end{table}

\newpage

\bibliographystyle{abbrvnat}{}
\bibliography{RoughHeston}

\end{document}